\documentclass[12pt,reqno]{amsart}
\makeatletter
\@namedef{subjclassname@2020}{%
  \textup{2020} Mathematics Subject Classification}
\makeatother
\usepackage[a4paper, margin=1.25in]{geometry}
\usepackage{amssymb,amsmath,amsthm,amscd,epsfig}
\usepackage[linktocpage=true]{hyperref}
\usepackage[english]{babel}
\usepackage{mathrsfs}
\usepackage{tikz-cd}
\usepackage{enumitem}
\usepackage{soul}
\usepackage{bm}
\usepackage{todonotes}
\newtheorem{thm}{Theorem}[section]
\newtheorem{lemma}[thm]{Lemma}
\newtheorem{theorem}[thm]{Theorem}
\newtheorem{proposition}[thm]{Proposition}
\newtheorem{corollary}[thm]{Corollary}
\theoremstyle{definition}
\newtheorem{note}[thm]{Note}
\newtheorem{question}[thm]{Question}
\newtheorem{conjecture}[thm]{Conjecture}
\newtheorem{example}[thm]{Example}
\newtheorem{remark}[thm]{Remark}
\newtheorem{mydef}[thm]{Definition}
\numberwithin{equation}{section}
\def\Z{\ifmmode{{\mathbb Z}}\else{${\mathbb Z}$}\fi}
\def\Q{\ifmmode{{\mathbb Q}}\else{${\mathbb Q}$}\fi}
\def\P{\ifmmode{{\mathbb P}}\else{${\mathbb P}$}\fi}
\def\R{\ifmmode{{\mathbb R}}\else{${\mathbb R}$}\fi}
\def\F{\ifmmode{{\mathbb F}}\else{${\mathbb F}$}\fi}
\def\O{\ifmmode{{\calO}}\else{${\calO}$}\fi}
\newcommand{\calA}{{\mathcal A}}

\newcommand{\calC}{{\mathcal C}}
\newcommand{\calD}{{\mathcal D}}
\newcommand{\calE}{{\mathcal E}}
\newcommand{\calF}{{\mathcal F}}

\newcommand{\calH}{{\mathcal H}}

\newcommand{\calO}{{\mathcal O}}
\newcommand{\calP}{{\mathcal P}}
\newcommand{\calQ}{{\mathcal Q}}

\newcommand{\calU}{{\mathcal U}}

\newcommand{\calW}{{\mathcal W}}
\newcommand{\calX}{{\mathcal X}}
\newcommand{\calY}{{\mathcal Y}}
\newcommand{\calZ}{{\mathcal Z}}
\DeclareMathOperator{\inv}{inv}
\DeclareMathOperator{\im}{im}

\DeclareMathOperator{\Spec}{Spec}
\DeclareMathOperator{\ev}{ev}
\DeclareMathOperator{\red}{red}
\DeclareMathOperator{\res}{res}
\DeclareMathOperator{\cores}{cores}
\DeclareMathOperator{\Bl}{Bl}
\DeclareMathOperator{\Br}{Br}
\newcommand{\A}{\mathbb{A}}
\newcommand{\injects}{\hookrightarrow}
\newcommand{\isom}{\simeq}
\let\originalleft\left
\let\originalright\right
\renewcommand{\left}{\mathopen{}\mathclose\bgroup\originalleft}
\renewcommand{\right}{\aftergroup\egroup\originalright}

\title{Weak approximation and the Hilbert property for Campana points}
\author{Masahiro Nakahara}
\thanks{The first-named author was supported by EPSRC grant EP/R021422/2. The second-named author was supported by a Heilbronn Research Fellowship.}
\address{Masahiro Nakahara, Department of Mathematics, University of Washington, Seattle, WA 98195, USA.}
\author{Sam Streeter}
\address{Sam Streeter, School of Mathematics, University of Bristol, Woodland Road, Bristol, BS8 1UG, UK.}
\begin{document}
\begin{abstract}
We study weak approximation and the Hilbert property for Campana points, both of importance in recent work on a Manin-type conjecture by Pieropan, Smeets, Tanimoto and V\'arilly-Alvarado. We show that weak weak approximation implies the Hilbert property for Campana points, and we exploit this to exhibit Campana orbifolds whose sets of Campana points are not thin.
\end{abstract}
\maketitle
\tableofcontents
\section{Introduction}
The theory of Campana points offers the enticing prospect of being able to interpolate between rational and integral points. With this prospect comes the hope that results and techniques from the study of rational points may be pulled into the Campana realm, from which they may shed light on the mysterious world of integral points. Further, the theory naturally lends itself to number-theoretic problems involving \emph{$m$-full} solutions of homogeneous equations. (Given $m \in \mathbb{Z}_{\geq 2}$, we say that $n \in \mathbb{Z}$ is $m$-full if $p \mid n$ implies $p^m \mid n$ for all primes $p$.) The study of solutions to equations with stipulations on prime decomposition is well-established within number theory. For instance, the Erd\H{o}s--Mollin--Walsh conjecture asserts that the system of linear equations $x+1 = y$, $y+1=z$ has no solutions in $m$-full numbers for $m=2$ (thus for any $m \geq 2$). Despite their arithmetic appeal, questions regarding the geometric abundance and distribution of Campana points remain largely unanswered outside of the case of curves: for curves, the orbifold version of the Mordell conjecture has been proved over function fields in characteristic zero by Campana (\cite{CAM05}) and in positive characteristic by Kebekus, Pereira and Smeets \cite{KPS}, while Smeets showed in \cite[Appendix]{SME} that the analogous result over number fields is implied by the $abc$ conjecture. In particular, the question of whether the Campana points of an orbifold are thin is open even for the most elementary orbifold structures on projective space. Thin sets play an important role in modern refinements of Manin's conjecture originally due to Peyre \cite{PEY}, in which the exceptional set is thin (see \cite[Thm.~1.3]{LR} and \cite[Thm.~1.1]{BHB} for examples and \cite[\S5]{LST} for an overview). Recently, a Manin-type conjecture for Campana points was formulated by Pieropan, Smeets, Tanimoto and V\'arilly-Alvarado \cite[Conj.~1.1]{PSTVA}, in which the exceptional set is also thin. Work of the second author \cite[\S7]{STR} and of Shute \cite{SHU1, SHU2} suggests issues around the conjecture's predicted leading constant, as asymptotics without the removal of a thin set may produce a different leading constant, and it is not clear which thin set to remove in order to reconcile this difference. However, hope remains of a Campana analogue of the thin set formulation of Manin's conjecture, and the exponents in all asymptotics to date agree with the conjecture with those in \cite[Conj.~1.1]{PSTVA}, hence the study of thin sets of Campana points is well-motivated.

In this paper, we take a step in this direction by studying local and adelic Campana points and their relationship to their rational counterparts. In particular, we focus on \emph{Campana weak weak approximation} (Definition \ref{def:CWWA}(i)). This property is interesting in its own right, since it tells us whether the Campana points of an orbifold are ``equidistributed'' in some sense. Further, it relates, as it does for rational points, to the Hilbert property. We show that Campana weak weak approximation implies that the set of Campana points is not thin, i.e.\ that the orbifold in question satisfies the \emph{Campana Hilbert property} (Definition \ref{def:CHP}). We exploit this connection in order to verify the Campana Hilbert property for various orbifolds. We also generalise a version of the fibration method used in the study of the Hilbert property for surfaces in order to verify the Campana Hilbert property for del Pezzo surfaces with a rational line.
\subsection{Results}
Let $\left(X,D\right)$ be a Campana orbifold (Definition \ref{def:CO}) over a number field $K$. In particular, $D$ is a $\Q$-divisor on $X$ which we write as
$$
D = \sum_{\alpha \in \calA} \epsilon_{\alpha} D_{\alpha},
$$
where the $D_{\alpha}$ are prime divisors and $\epsilon_{\alpha} = 1 - \frac{1}{m_{\alpha}}$ for some $m_{\alpha} \in \Z_{\geq 2} \cup \{\infty\}$ (taking $\epsilon_\alpha = 1$ when $m_\alpha = \infty$).

Let $S$ be a finite subset of places of $K$ containing all of the archimedean places. Let $\calX$ be a flat proper $\O_S$-model for $X$ and let $\calD$ be the closure of $D$ in $\calX$. We denote the set of Campana points of $(\calX,\calD)$ by $(\calX,\calD)(\O_S)$ (Definition \ref{def:CP}). This is a subset of $X(K)$, and one can ask whether it is thin. If not, we say that $(\calX,\calD)$ satisfies the Campana Hilbert property (CHP).

Recall that a variety $Y$ over $K$ satisfies \emph{weak weak approximation} (WWA) if there exists a finite subset of places $S$ such that $Y(K)$ is dense in $\prod_{v \not\in S}Y\left(K_v\right)$. It was shown by Colliot-Th\'el\`ene and Ekedahl that weak weak approximation implies the Hilbert property \cite[Thm.~3.5.7]{SER}. 

Campana weak weak approximation (Definition \ref{def:CWWA}(i)), abbreviated to CWWA, is defined by replacing rational points and local points by their Campana analogues. Our first result is that the Campana analogue of Colliot-Th\'el\`ene and Ekedahl's result holds, i.e.\ CWWA implies CHP.
\begin{theorem} \label{cwwathm}
Let $\left(X,D\right)$ be a Campana orbifold over a number field $K$, let $S$ be a finite subset of places of $K$ containing all archimedean places and let $\left(\calX,\calD\right)$ be an $\calO_S$-model of $\left(X,D\right)$. If $\left(\calX,\calD\right)$ satisfies Campana weak weak approximation and $(\calX,\calD)(\O_S)\neq\emptyset$, then it has the Campana Hilbert property.
\end{theorem}
With this relation in mind, we devote the rest of the paper to Campana weak weak approximation and the stronger property of \emph{Campana weak approximation} (Definition \ref{def:CWWA}(ii)).

We first turn to the case when $X=\P^n_K$ for some number field $K$, $\calA=\{0,\ldots,r\}$ and the $D_i$ are hyperplanes. We say that an orbifold $(X,D)$ is log Fano when $-(K_X+D)$ is ample. Then the orbifold $(\P^n,D)$ is log Fano precisely when
\begin{equation}\label{eqn:logFano}
\sum_{i=0}^r \left(1-\frac{1}{m_i}\right) <  n+1.
\end{equation}
A variant of the following question was asked in \cite[Question~3.8]{PSTVA}.
\begin{question}\label{question}
Suppose that $(\P^n,D)$ is log Fano. Given a finite set of places $S \subset \Omega_K$ containing all archimedean places, is the set of Campana points $(\P^n_{\O_S},\calD)(\O_S)$ non-thin?
\end{question}
Asymptotics for Campana points of bounded height on such orbifolds over $\Q$ were studied by Browning and Yamagishi \cite{BY}. They considered the case where $r=n+1$ (the first non-trivial case for counting) and $D_i$ is the coordinate hyperplane $x_i = 0$ for $0 \leq i \leq n$, while $D_{n+1}$ is given by $\sum_i x_i=0$. Using their asymptotics, they showed that a certain collection of thin sets are not enough to cover $(\P^n_{\Z},\calD)(\Z)$ \cite[Thm.~1.3]{BY} for $n$ large relative to the $m_i$. In other words, the number of Campana points grows faster than those in these thin sets. 

Our approach to answer Question \ref{question} is to study the distribution of Campana points through Campana weak weak approximation. We give an affirmative answer when $r \leq n$ and $H_0,\dots,H_r$ are in general linear position (Definition \ref{def:glp}) provided that there exist local Campana points over each place $v \not\in S$. Note that in this case, \eqref{eqn:logFano} is satisfied for all choices of $m_i$. In fact, we prove the stronger result that Campana weak approximation (Definition \ref{def:CWWA}(ii)) holds. We also prove a partial result when $r=n+1$, conditional on a conjecture of Colliot-Th\'el\`ene on weak approximation with Brauer--Manin obstruction (\ref{def:WABMO}) for geometrically rationally connected varieties, which needs only to be assumed for certain Fano diagonal hypersurfaces.
\begin{theorem} \label{thm:bmo} 
Let $\left(\P^n,D\right)$ be as above with $r \leq n+1$, and suppose that $H_0,\dots,H_r$ are in general linear position.
\begin{enumerate}[label=(\roman*)]
\item If $r \leq n$, then $\left(\P^n_{\calO_S},\calD\right)$ satisfies Campana weak approximation.
\item Let $r = n+1$. Assume that all $m_i$ are equal to some integer $m$. If weak approximation with Brauer--Manin obstruction holds for smooth diagonal hypersurfaces of degree $m$ in $\mathbb{P}^{n+1}$, then $\left(\P^n_{\calO_S},\calD\right)$ satisfies Campana weak weak approximation. Further, if $\left(m,n\right) \neq \left(3,2\right)$, then $\left(\P^n_{\calO_S},\calD\right)$ satisfies Campana weak approximation.
\end{enumerate}
\end{theorem}
Since weak approximation holds for quadrics \cite[Thm.~2.2.1]{HAR}, the assumption of weak approximation with Brauer--Manin obstruction in the second part of Theorem \ref{thm:bmo} is satisfied when $m = 2$ (for any $n$).

Similarly, results of Wooley \cite[Thm.~15.6]{WOO19} and of Birch \cite[Thm.~3]{BIR} on the Hasse principle collectively establish (as a byproduct) weak approximation for diagonal hypersurfaces of degree $m$ in $\mathbb{P}^{n+1}$ for $n \geq \min\{m\left(m+1\right) - 1,2^m-1\}$, which improves upon Skinner's result for all $m \geq 3$. (For details on the derivation of weak approximation from circle method proofs of the Hasse principle, see \cite[\S5]{SKI}.) This bound can be improved to $n \geq \min\{m\left(m+1\right) - 2,2^m-1\}$ for $K = \mathbb{Q}$ using \cite[Thm.~1.4]{BY}, and one can make a further small improvement using methods from \cite{WOO92} and \cite{WOO95}. We therefore deduce the following result.
\begin{corollary} \label{cor:cwa}
Let $\left(\P^n,D\right)$ be as in the second part of Theorem \ref{thm:bmo}. Campana weak approximation for $\left(\P^n_{\O_S},\calD\right)$ holds in the following cases. 
\begin{enumerate}[label=(\roman*)]
\item $m=2$.
\item $n \geq \min\{m\left(m+1\right)-1,2^m-1\}$.
\item $K = \mathbb{Q}$ and $n \geq \min\{m\left(m+1\right) - 2,2^m-1\}$.
\end{enumerate}
\end{corollary}
\begin{note}
One may view Corollary \ref{cor:cwa} as a generalisation of \cite[Thm.~1.3]{BY} in the case where all of the $m_i$ are equal. Indeed, loc.\ cit.\ shows that certain thin sets are not enough to cover the Campana points when $n \geq m\left(m+1\right) - 2$, whereas our results show that no collection of thin sets can do this.
\end{note}
Next we develop a fibration result for Campana points, based on the work of Bary-Soroker, Fehm and Petersen \cite[Thm.~1.1]{BSFP}. As an application, we obtain the following result for del Pezzo surfaces.
\begin{theorem} \label{thm:dp}
Let $X$ be a del Pezzo surface of degree $d$ over a number field $K$, let $L \subset X$ be a rational line on $X$, and let $m \geq 2$ be an integer. Put $D = \left(1-\frac{1}{m}\right)L$. Let $\left(\calX,\calD_m\right)$ be an $\calO_S$-model of $\left(X,D\right)$ for some finite set of places $S$ of $K$ containing all archimedean places. Suppose that one of the following holds.
\begin{enumerate} [label=(\roman*)]
\item $d=4$ and $X$ has a conic fibration.
\item $d=3$.
\item $d=2$ and $X$ has a conic fibration.
\end{enumerate}
Then $\left(\calX,\calD\right)$ has the Campana Hilbert property.
\end{theorem}

The analogous result for rational points was proved in \cite[Thm.~1.4]{STR}. To prove Theorem \ref{thm:dp}, we borrow ideas from \cite{STR} by using the induced double conic bundle structure. We first establish a lemma relating non-thin sets of Campana points and blowing up. This allows us to reduce, as in \emph{ibid.}, to case (iii), with extra difficulty coming from the fact that we must consider Campana points on $\P^1$ with a degree-$2$ point as the orbifold divisor. Hence, the proof relies on the application of a crucial result (Proposition \ref{prop:deg2P1}) establishing Campana Hilbert property for such orbifolds in Section \ref{sec:proj}.
\subsection{Conventions}
\subsubsection*{Number theory}
We denote by $\Omega_K$ the set of places of a number field $K$. Given $v \in \Omega_K$, we denote the completion of $K$ at $v$ by $K_v$ and choose the absolute value $|\cdot|_v$ on $K_v$ given by $|x|_v = |N_{K_v/\mathbb{Q}_p}\left(x\right)|_p$, where $|\cdot|_p$ is the usual absolute value on $\mathbb{Q}_p$ and $v$ lies over $p \in \Omega_\mathbb{Q}$. For $v$ non-archimedean, we denote by $v: K \rightarrow \mathbb{Z} \cup \{\infty\}$ the associated discrete valuation, and by $\calO_v$ the $v$-adic integers. We denote by $\pi_v$ and $\F_v$ a uniformiser and the residue field for $\mathcal{O}_v$ respectively, and we write $q_v$ for the size of $\F_v$. When working over a number field $K$, we denote by $S_{\infty}$ the set of all archimedean places. Given a finite set $S \subset \Omega_K$ containing $S_{\infty}$, we denote by $\calO_S = \{\alpha \in K: v\left(\alpha\right) \geq 0 \textrm{ for all } v\not\in S\}$ the ring of $S$-integers of $K$. We write $\mathcal{O}_K = \mathcal{O}_{S_{\infty}}$.
\subsubsection*{Geometry}
We denote by $\A_R^n$ and $\P_R^n$ the affine and projective $n$-spaces over a ring $R$ respectively, omitting the subscript if $R$ is clear. A variety over a field $F$ is a geometrically integral separated scheme of finite type over $F$. A curve is a variety of dimension $1$, and a surface is a variety of dimension $2$. Given a variety $X$ over $F$ and a field extension $E/F$, we write $X_E = X \times_{\Spec F} \Spec E$; when $F$ is a number field $K$ and $v \in \Omega_K$, we write $X_v = X \times_{\Spec K} \Spec K_v$. Given finite sets of places $S \subset T \subset \Omega_K$ and a scheme $\calX$ over $\Spec \O_S$, we write $\calX_T = \calX \times_{\Spec \O_S} \Spec \O_T$. Given a scheme $\calX$ over $\O_S$ and a place $v \not\in S$, we write $\calX_v = \calX \times_{\Spec \O_S} \Spec \O_v$.

We use the following definitions throughout the paper.
\begin{mydef}
Let $X$ be a smooth projective surface. A \emph{conic fibration} of $X$ is a morphism $\pi \colon X \rightarrow \P^1$ such that all fibres of $\pi$ are isomorphic to plane conics.
\end{mydef}
\begin{mydef}
Let $X$ be a del Pezzo surface. A \emph{line} on $X$ is a curve $E \subset X$ (necessarily irreducible) such that $E^2 = E \cdot K_X = -1$.
\end{mydef}
\begin{mydef} \label{def:glp}
Given a hyperplane $H : \sum_{i=0}^n a_i x_i = 0$ in $\mathbb{P}^n$ over a field $K$, we call the point $[a_0:\dots:a_n] \in \mathbb{P}^n(K)$ the \emph{normal} to $H$. We say that the hyperplanes $H_0,\dots,H_r$ are in \emph{general linear position} if their normals are in general linear position, i.e.\ no $k$ of them lie in a $\left(k-2\right)$-dimensional linear subspace for $k=2,3,\dots,n+1$.
\end{mydef}
\subsection*{Acknowledgements}
We would like to thank Tim Browning, Daniel Loughran, Harkaran Uppal and the authors of \cite{PSTVA} for many useful comments and discussions. In particular, we would like to thank Sho Tanimoto for suggesting weak weak approximation for verifying the Hilbert property for Campana points and Marta Pieropan for her comments on an earlier draft. We also thank Trevor Wooley for pointing out the improvements for weak approximation on diagonal hypersurfaces used in Corollary~\ref{cor:cwa}, and we thank David Bourqui and Gregory Sankaran for corrections suggested when the paper appeared in the second-named author's PhD thesis. Lastly, we thank the anonymous referee for many thoughtful suggestions which have vastly improved the exposition. The first-named author was supported by EPSRC grant EP/R021422/2, while the second-named author was supported by an EPSRC PhD studentship and the Heilbronn Institute for Mathematical Research.
\section{CWWA and CHP}
In this section we define Campana orbifolds as well as local, rational and adelic Campana points.
We then define Campana weak weak approximation and the Campana Hilbert property, and we prove Theorem \ref{cwwathm}.
\subsection{Campana points} \label{subsectionCP}
The concepts of Campana orbifolds and orbifold rational points come from Campana's theory of ``orbifoldes g\'eom\'etriques'' (see e.g.\ \cite{CAM}) and are also developed in \cite{ABR}. We employ the definitions given in \cite{STR2}.
\begin{mydef} \label{def:CO}
A \emph{Campana orbifold} over a field $F$ is a pair $\left(X,D\right)$ consisting of a proper, normal variety $X$ over $F$ and an effective Cartier $\Q$-divisor
$$
D = \sum_{\alpha \in \calA} \epsilon_{\alpha} D_{\alpha}
$$
on $X$, where the $D_{\alpha}$ are prime divisors and $\epsilon_{\alpha} = 1 - \frac{1}{m_{\alpha}}$ for some $m_{\alpha} \in \Z_{\geq 2} \cup \{\infty\}$, taking $\epsilon_\alpha = 1$ when $m_\alpha = \infty$.
We define the \emph{support} of the $\Q$-divisor $D$ to be
$$
D_{\textrm{red}} = \sum_{\alpha \in \calA} D_{\alpha}.
$$
\end{mydef}
Let $\left(X,D\right)$ be a Campana orbifold over a number field $K$, and let $S$ be a finite set of places of $K$ containing $S_{\infty}$.
\begin{mydef}
A \emph{model} of $\left(X,D\right)$ over $\calO_S$ is a pair $\left(\calX,\calD\right)$, where $\calX$ is a flat proper model of $X$ over $\calO_S$ (i.e.\ a flat proper $\calO_S$-scheme together with an isomorphism $\calX_{\left(0\right)} \xrightarrow{\sim} X$), $\calD_{\alpha}$ is the Zariski closure of $D_{\alpha}$ in $\calX$ and $\calD = \sum_{\alpha \in \calA}\epsilon_{\alpha} \calD_{\alpha}$.
For a place $v \not\in S$, we denote by $\calD_{\alpha_v}$, $\alpha_v \in \calA_v$ the irreducible components of $\calD_{\textrm{red}}$ over $\Spec \calO_v$, and we write $\alpha_v \mid \alpha$ if $\calD_{\alpha_v} \subset \calD_{\alpha}$.
\end{mydef}
Let $\left(\calX,\calD\right)$ be a model for $\left(X,D\right)$ over $\calO_S$. For $v \not\in S$ a place of $K$, a point $P \in X\left(K_v\right)$ induces a point $\calP_v \in \calX\left(\calO_v\right)$ by the valuative criterion of properness \cite[Thm.~II.4.7]{HART}.
\begin{mydef}
Let $P \in X\left(K\right)$ and take a place $v \not\in S$. For each $\alpha_v \in \mathcal{A}_v$, we define the \emph{local intersection multiplicity} $n_v\left(\mathcal{D}_{\alpha_v},P\right)$ of $\mathcal{D}_{\alpha_v}$ and $P$ at $v$ to be $+\infty$ if $P\in D_{\alpha_v}$ and the colength of the $\O_v$-ideal corresponding to the closed subscheme $\mathcal{P}_v^*\mathcal{D}_{\alpha_v}$ of $\Spec \mathcal{O}_v$ otherwise.
\end{mydef}
\begin{mydef}\label{def:CP}
Let $v \not\in S$. We say that $P \in X\left(K_v\right)$ is a \emph{Campana $\calO_v$-point} of $\left(\calX,\calD\right)$ if the following implications hold for all $\alpha \in \calA$.
\begin{enumerate}[label=(\roman*)]
\item If $\epsilon_{\alpha} = 1$ (meaning $m_{\alpha} = \infty$), then $n_v\left(\calD_{\alpha_v},P\right) = 0$ for all $\alpha_v\mid \alpha$.
\item If $\epsilon_\alpha \neq 1$ and $n_v\left(\calD_{\alpha_v},P\right) > 0$ for some $\alpha_v \mid \alpha$, then
$$
n_v\left(\calD_{\alpha_v},P\right) \geq \frac{1}{1-\epsilon_{\alpha}},
\textrm{ i.e. }
n_v\left(\calD_{\alpha_v},P\right) \geq m_{\alpha}.
$$
\end{enumerate}
We denote the set of Campana $\calO_v$-points of $\left(\calX,\calD\right)$ by $\left(\calX,\calD\right)\left(\calO_v\right)$.
We say that $P \in X\left(K\right)$ is a \emph{Campana $\calO_S$-point} of $\left(\calX,\calD\right)$ if its image in $X\left(K_v\right)$ belongs to $\left(\calX,\calD\right)\left(\calO_v\right)$ for all $v \not\in S$.
We set $\left(\calX,\calD\right)\left(\calO_{v'}\right) = X\left(K_{v'}\right)$ for any place $v' \in S$.

For any finite set of places $T \subset \Omega_K$, we define the set of \emph{Campana $\A_K^T$-points} of $(\calX,\calD)$ to be $(\calX,\calD)(\A_K^T) = \prod_{v \not\in T}(\calX,\calD)(\calO_v)$, i.e.\
$$
(\calX,\calD)(\A_K^T) = \prod_{v \not\in S \cup T}(\calX,\calD)(\calO_v) \times \prod_{v \in S \setminus T}X\left(K_v\right).
$$
We define the set of \emph{adelic Campana points} of $(\calX,\calD)$ to be
$$
(\calX,\calD)(\A_K) = \prod_{v \not\in S}(\calX,\calD)(\calO_v) \times \prod_{v \in S}X\left(K_v\right).
$$
Both $(\calX,\calD)(\A^T_K)$ and $(\calX,\calD)(\A_K)$ are equipped with the subspace topology inherited from $X(\A^T_K)$ and $X(\A_K)$.
\end{mydef}
\begin{note}\label{note:open}
As observed in \cite[\S6.2]{PSTVA}, intersection multiplicity is locally constant on $\left(X \setminus D\right)\left(K_v\right)$ for all places $v \not\in S$, whence it follows that $\left(\calX,\calD\right)\left(\O_v\right)$ is both closed and open in $X\left(K_v\right)$ for all $v \in \Omega_K$.
\end{note}
\begin{remark} \label{rem:ratintcamp}
The rational points of $X$ are precisely the Campana points when $D$ is the zero divisor, and integral points (with respect to $D$) correspond to every coefficient of the orbifold divisor being $1$, thus are those points whose induced $\O_S$-point ``avoids'' $\calD_{\red}$. The Campana points are the rational points $P \in X \left(K\right)$ which either avoid the ($v$-adic) irreducible components of the $\calD_\alpha$ or meet them in high multiplicity.
\end{remark}
\begin{mydef} \label{def:CWWA}
Let $\left(X,D\right)$ be a Campana orbifold over $K$ with $\O_S$-model $\left(\calX,\calD\right)$.
\begin{enumerate}[label=(\roman*)]
\item We say that $\left(\calX,\calD\right)$ satisfies \emph{Campana weak weak approximation (CWWA)} if there exists a finite set $T$ of places of $K$ such that $(\calX,\calD)(\O_{S})$ is dense in $\left(\calX,\calD\right)\left(\A_K^T\right)$.
\item We say that $(\calX,\calD)$ satisfies \emph{Campana weak approximation (CWA)} if the set $(\calX,\calD)(\O_S)$ is dense in $(\calX,\calD)(\A_K)$.
\end{enumerate}
\end{mydef}
Observe that CWA implies CWWA (taking $T = \emptyset$). Note also that one recovers weak approximation for $X$ and strong approximation for $X \setminus D_{\red}$ from CWA upon taking $D = 0$ and $D = D_{\red}$ respectively. The following lemma tells us that, if CWA is preserved upon enlarging $S$, then it is independent of the choice of model.
\begin{lemma}\label{lem:twoModels}
Let $K$ be a number field and let $S,T$ be finite subsets of places of $K$ where $S$ contains $S_{\infty}$. Let $(X,D)$ be a Campana orbifold over $K$ with $\calO_S$-models $(\calX,\calD)$ and $(\calY,\calE)$. Suppose that for all finite subsets $S\subset S'\subset \Omega_K$, the set $(\calX_{S'},\calD_{S'})(\O_{S'})$ is dense in $(\calX_{S'},\calD_{S'})(\A^{T}_K)$. Then $(\calY,\calE)(\O_S)$ is dense in $(\calY,\calE)(\A^T_K)$.
\end{lemma}
\begin{proof}
By spreading out (see \cite[\S3.2]{POO}), there exists a finite set of places $S'$ containing $S$ such that
$$(\calX_{S'},\calD_{S'})\isom(\calY_{S'},\calE_{S'})$$
over $\Spec\O_{S'}$. The hypothesis on $(\calX,\calD)$ and the above isomorphism imply that $(\calY_{S'},\calE_{S'})(\O_{S'})$ is dense in $(\calY_{S'},\calE_{S'})(\A_K^T)$. Since the local Campana points are open at each place (see Note \ref{note:open}), $(\calY,\calE)(\A^T_K)$ is open inside $(\calY_{S'},\calE_{S'})(\A^T_K)$. Hence, we see that $(\calY,\calE)(\O_S)=(\calY_{S'},\calE_{S'})(\O_{S'})\cap (\calY,\calE)(\A^T_K)$ is dense in $(\calY,\calE)(\A^T_K)$.
\end{proof}
\subsection{Hilbert property}
In this section we introduce thin sets and the Hilbert property (an excellent reference for which is \cite{SER}) and we prove Theorem \ref{cwwathm}. As mentioned in the introduction, the motivation for studying thin sets of Campana points is twofold: to understand from a geometric perspective their ubiquity, and to better understand the exceptional set in \cite[Conj.~1.1]{PSTVA}. An assumption forced by this conjecture (see \cite[\S3.4]{PSTVA}) is that the set of Campana points itself (on the orbifold of interest) is not thin. However, little is known regarding thin sets of Campana points and integral points aside from the result of Browning and Yamagishi \cite[Thm.~1.3]{BY} mentioned earlier and work of Coccia on integral points on affine cubic surfaces \cite{C}.

Let $V$ be a variety defined over a field $F$ of characteristic zero, and let $A \subset V\left(F\right)$.
\begin{mydef}\label{def:HP}
We say that $A$ is of \emph{type I} if $A \subset W\left(F\right)$ for some proper closed subset $W$ of $V$, i.e.\ $A$ is not dense in $V$.

We say that $A$ is of \emph{type II} if $A \subset \phi\left(V'\left(F\right)\right)$ for $V'$ a variety and $\phi \colon V' \rightarrow V$ a generically finite dominant morphism of degree $\geq 2$.

We say that $A$ is \emph{thin} if it is contained in a finite union of subsets of $V(F)$ of types I and II.
\end{mydef}
\begin{note}
One may assume without loss of generality in Definition \ref{def:HP} that $V'$ is normal and $\pi$ is finite.
\end{note}
We say that $V$ satisfies the \emph{Hilbert property} if $V(F)$ is not thin. Motivated by this, we define the following analogue for Campana points. Let $\left(X,D\right)$ be a Campana orbifold over a number field $K$ with $\calO_S$-model $\left(\calX,\calD\right)$ for some finite set of places $S$ of $K$ containing $S_{\infty}$.
\begin{mydef} \label{def:CHP}
We say that $\left(X,D\right)$ satisfies the \emph{Campana Hilbert property (CHP)} if the set $\left(\calX,\calD\right)\left(\calO_S\right) \subset X\left(K\right)$ is not thin.
\end{mydef}
\subsection{Proof of Theorem \ref{cwwathm}}
We now prove that CWWA implies CHP, provided that the set of Campana points is non-empty. The proof is a direct adaptation of Serre's proof of a theorem (namely \cite[Thm.~3.5.3]{SER}) which plays a crucial step in the case of rational points.
\begin{proof} [Proof of Theorem \ref{cwwathm}]
The statement of the thoerem is equivalent to saying if $(\calX,\calD)$ does not have the Campana Hilbert property and $(\calX,\calD)(\calO_S)\neq\emptyset$, then it does not satisfy Campana weak weak approximation. We will prove a stronger version of this statement: if $A \subset \left(\calX,\calD\right)\left(\calO_S\right)$ is thin and $T_0 \subset \Omega_K$ is finite, then there is another finite subset $T \subset \Omega_K$ disjoint with $T_0$ such that $A$ is not dense in $\prod_{v \in T}\left(\calX,\calD\right)\left(\calO_v\right)$. Indeed, the original statement of the theorem is the case where $A=(\calX,\calD)(\O_S)$. Note that, if the stronger statement is true for the thin subsets $A_1$ and $A_2$, then it is also true for $A_1 \cup A_2$ (cf.\ \cite[Proof~of~Thm.~3.5.3]{SER}). Then it suffices to prove it when $A$ is type I and when $A$ is type II. Denote by $Y$ the proper closed subset of singular points of $X$.

Let $A \subset W\left(K\right) \cap \left(\calX,\calD\right)\left(\calO_S\right)$ for $W \subset X$ a proper closed subset of $X$. Denote by $\calW$ and $\calY$ the Zariski closures in $\calX$ of $W$ and $Y$ respectively. Since $W\left(K_v\right) \subset X\left(K_v\right)$ is closed, it suffices to find $v \not\in T_0$ such that $\left(\calX,\calD\right)\left(\calO_v\right) \setminus W\left(K_v\right) \neq \emptyset$. Taking $q_v$ sufficiently large, applying the Lang--Weil estimate \cite[Thm.~3.6.1]{SER} (which we may do as $X$ is normal with a rational point, hence geometrically irreducible) and generic smoothness of $X \setminus Y$ shows that there exists a smooth point $Q \in \calX_v\left(\F_v\right) \setminus \left(\calD_v\left(\F_v\right) \cup \calW_v\left(\F_v\right) \cup \calY_v\left(\F_v\right)\right)$, so $Q$ lifts to a point $P \in \left(\calX\setminus\calD\right)\left(\calO_v\right) \setminus W\left(K_v\right) \subset \left(\calX,\calD\right)\left(\calO_v\right) \setminus W\left(K_v\right)$.

Now let $A \subset \pi\left(W\left(K\right)\right) \cap \left(\calX,\calD\right)\left(\calO_S\right)$ for $W$ a normal $K$-variety and $\pi \colon W \rightarrow X$ a finite morphism of degree $ \geq 2$. Since $\pi$ is finite, we deduce that $\pi\left(W\left(K_v\right)\right) \subset X\left(K_v\right)$ is closed, so it suffices to find $v \in \Omega_K$ with $\left(\calX,\calD\right)\left(\calO_v\right) \setminus \pi\left(W\left(K_v\right)\right) \neq \emptyset$. By combining \cite[Thm.~3.6.2]{SER} and another application of the Lang--Weil estimate along with generic smoothness of $X \setminus Y$, there exists a smooth point $Q \in \calX_v\left(\F_v\right) \setminus \left(\calD_v\left(\F_v\right) \cup \pi\left(\calW_v\left(\F_v\right)\right) \cup \calY_v\left(\F_v\right)\right)$ for all $v$ totally split in the algebraic closure of $K$ in the Galois closure $K\left(W\right)^{\textrm{gal}}/K\left(V\right)$ with $q_v$ sufficiently large. Taking such $v$ and $Q$, we may lift $Q$ to a point $P \in \left(\calX\setminus\calD\right)\left(\calO_v\right) \setminus \pi\left(W\left(K_v\right)\right) \subset \left(\calX,\calD\right)\left(\calO_v\right) \setminus \pi\left(W\left(K_v\right)\right)$.
\end{proof}
\begin{remark}
The above proof actually gives a slightly stronger result: given a thin set $A \subset X\left(K\right)$ and a finite subset $T_0 \subset \Omega_K$, there exists a finite set of places $T \subset \Omega_K$ disjoint with $T_0$ such that the image of $A$ is not dense in $\prod_{v \in T}\left(\calX \setminus \calD\right)\left(\O_v\right)$. In particular, in order to verify that a Campana orbifold has CHP, it suffices to show that, for all finite sets of places $T_0 \supset S$, the closure of $\left(\calX,\calD\right)\left(\O_S\right)$ in $\prod_{v \not\in T_0}\left(\calX,\calD\right)\left(\O_v\right)$ contains $\prod_{v \not\in T_0}\left(\calX \setminus \calD\right)\left(\O_v\right)$.
\end{remark}
\section{Projective space}\label{sec:proj}
Let $K$ be a number field, and let $S \subset \Omega_K$ be a finite set containing $S_{\infty}$. In this section we consider approximation properties for Campana $\calO_S$-points on orbifolds $(\P^n,D)$, where
$$
D =\sum_{i=0}^r \left(1-\frac{1}{m_i}\right)D_i
$$
for $D_0,\dots,D_r$ hyperplanes in general linear position and $\{m_0,\dots,m_r\} \subset \mathbb{Z}_{\geq 2}$.
\subsection{Reduction to coordinate hyperplanes}\label{sec:reduction}
Assume $r\leq n+1$. Define the hyperplanes
$$
H_i : x_i=0\quad \textrm{for } 0 \leq i \leq n,\quad H_{n+1} \colon \sum_{i=0}^n x_i=0,
$$
and define the $\mathbb{Q}$-divisor
\begin{equation} \label{eqn:n+2}
H = \sum_{i=0}^r \left(1-\frac{1}{m_i}\right)H_i.
\end{equation}
Let $D_0,\ldots,D_r\subset\P^n$ denote arbitrary hyperplanes in general linear position. By applying a projective transformation, we obtain an isomorphism
$$
f\colon (\P^n,D) \xrightarrow{\sim} (\P^n,H).
$$
Denote by $\calD$ and $\mathcal{H}$ the closures of $D$ and $H$ in $\P_{\O_S}^n$ respectively. By Lemma \ref{lem:twoModels}, if $(\P^n_{\O_S},\calH)$ satisfies CWA (respectively, if $(\P^n_{\O_S},\calH)(\O_S)\subset (\P^n_{\O_S},\calH)(\A^T_K)$ is dense) for arbitrary $S$, then $(\P^n_{\O_S},\calD)$ also satisfies CWA (respectively, then $(\P^n_{\O_S},\calD)(\O_S)\subset (\P^n_{\O_S},\calD)(\A^T_K)$ is dense). Hence, when proving CWA or CWWA for arbitrary $S$, we can reduce to the case $D=H$. In this case, we have the following concrete description for the set of Campana points. Write $P = [x_0:\cdots:x_n]$ and set $x_{n+1}=\sum x_i$. For any $v\notin S$,
\begin{equation}\label{eqn:vadiccampana}
(\P^n_{\calO_S},\calH)(\O_v)=\Bigl\{P \in\P^n(K_v):v(x_i)-\min \{v(x_j)\} \in \Z_{\geq m_i} \cup \{0, \infty\}, 0 \leq i \leq r\Bigr\}
\end{equation}
and
\begin{equation}\label{eqn:campana}
(\P^n_{\calO_S},\calH)(\O_S)=\P^n(K)\cap\left(\prod_{v\notin S}(\P^n_{\O_S},\calH)(\O_v)\times\prod_{v\in S}\P^n(K_v)\right).
\end{equation}
\subsection{Local solubility}
Before investigating CWA, let us consider the question of local solubility for Campana points on $\left(\P^n,D\right)$, i.e.\ whether $(\P^n_{\O_S},\calD)(\A_K)$ is non-empty. If there are no adelic Campana points, then CWA holds trivially, but CHP clearly fails. In fact, this question is more delicate than one might expect. Let us give two illustrative examples.
\begin{example}
Suppose $D = H$, with $H$ as in \eqref{eqn:n+2}. Then it is clear that $[1,\ldots,1]\in\P^n(K)$ is a Campana $\O_K$-point according to \eqref{eqn:vadiccampana} and \eqref{eqn:campana}.
\end{example}
\begin{example} \label{ex:ncp}
Consider the hyperplanes $D_i\subset\P^5_\Q$ defined by $f_i=0$, where
\begin{equation*}
\begin{aligned}
f_0& =x_0+2x_1+4g_0, \quad & f_1 &=5x_0+4x_1+4g_1, \\
f_2& =2x_0+x_1+4g_2,\quad & f_3 &=4x_0+5x_1+4g_3, \\
f_4& =x_0+x_1+4g_4,\quad & f_5 &=x_0+3x_1+4g_5,
\end{aligned}
\end{equation*}
and $g_i\in\Z[x_2,x_3,x_4,x_5]$ are linear forms chosen so that the $f_i$ are linearly independent. Set $m_i = m \geq 2$ for each $i=0,\dots,r$. Since there is no $\mathbb{Q}_2$-point lying on every $D_i$, we may choose $m$ large enough so that $f_0=\cdots=f_5=0$ has no solution in $\P^5(\Z/2^m\Z)$. Choosing such $m$, we will show that $\left(\P^5_{\Z},\calD\right)\left(\Z_2\right) = \emptyset$.

Suppose that $P = [a_0:\dots:a_5] \in \left(\P^5_{\Z},\calD\right)\left(\Z_2\right)$, with $a_0,\dots,a_5$ coprime $2$-adic integers. Note that modulo $2$, each hyperplane $D_i$ reduces to one of the following:
$$
x_0=0,\quad x_1=0,\quad x_0+x_1=0.
$$
Clearly any point in $\P^5\left(\F_2\right)$ must lie in one of the three hyperplanes above. Let $P_2$ denote the reduction of $P$ modulo $2$. Suppose that $P_2$ lies on $x_0=0$, i.e.\ that $2 \mid a_0$. According to the equations, this means $n_2(\calD_0,P),n_2(\calD_1,P)>0$. Since $P \in \left(\P^5_{\Z},\calD\right)\left(\Z_2\right)$, we must then have $n_2(\calD_0,P),n_2(\calD_1,P) \geq m \geq 2$. Since $f_1-f_0 \equiv 2x_1\pmod 4$, this means that $2 \mid a_1$ as well, i.e.\ $P_2$ lies on $x_1 = 0$, so $n_2\left(\calD_2,P\right), n_2\left(\calD_3,P\right) > 0$. Since $P_2$ lies on $x_0 = 0$ and $x_1 = 0$, it also lies on $x_0 + x_1$, hence $n_2\left(\calD_4,P\right), n_2\left(\calD_5,P\right) > 0$. We conclude that $n_2(\calD_i,P)>0$ for all $i$, hence we must have $n_2\left(\calD_i,P\right) \geq m$ for all $i$. However, since no point in $\P^5(\Z_2)$ lies on all hyperplanes $D_i$ modulo $2^m$, this is impossible. We conclude that $P$ is not a Campana point. The same argument works if $P_2$ were to lie on either $x_1=0$ or $x_0+x_1=0$, and so we conclude that $\left(\P^5_{\Z},\calD\right)\left(\Z_2\right) = \emptyset$.
\end{example}
Example \ref{ex:ncp} shows how things can go wrong when $q_v$ is small compared to $n$. We now give some sufficient conditions for the existence of a $v$-adic Campana point.
\begin{corollary}
Let $D_0,\dots,D_r\subset \P^n$ be hyperplanes in general linear position. Then $(\P^n_{\O_S},\calD)(\O_v)\neq\emptyset$ if one of the following conditions is satisfied:
\begin{enumerate}[label=(\roman*)]
\item $r<n$;
\item The reductions $(\calD_0)_{\F_v},\ldots,(\calD_r)_{\F_v}$ are linearly independent;
\item $q_v \geq n$.
\end{enumerate}
\end{corollary}
\begin{proof}
If $r<n$, then choose $P\in \left(D_0\cap\cdots\cap D_r\right)(K_v)$. Then $P\in(\P^n_{\O_S},\calD)(\O_v)$ by definition. For the rest of the proof, we assume that $r=n$.

If $(\calD_0)_{\F_v},\ldots,(\calD_n)_{\F_v}$ are linearly independent, then choose a $v$-adic point $P\in \left(D_0\cap\cdots\cap D_{n-1}\right)(K_v)$. By assumption, the reduction of $P$ modulo $v$ does not lie on $(\calD_n)_v$. Hence $n_v(\calD_i,P)=\infty$ if $i\leq n-1$ and $0$ if $i=n$. It follows that $P\in(\P^n_{\O_S},\calD)(\O_v)$.

Now assume $q_v\geq n$. If $\left(\mathcal{D}_0\right)_{\F_v} = \left(\mathcal{D}_i\right)_{\F_v}$ for some $i \geq 1$, then 
$$
|\mathbb{P}^n\left(\mathbb{F}_v\right)| = \frac{q_v^{n+1}-1}{q_v-1} > n\frac{q_v^n-1}{q_v-1} > |\left(\calD_1\cup\cdots\cup\calD_n\right)(\F_v)| = |\left(\calD_0\cup\cdots\cup\calD_n\right)(\F_v)|,
$$
hence there exists $P_v \in \mathbb{P}^n\left(\mathbb{F}_v\right) \setminus \bigcup_{i=0}^n \mathcal{D}_i\left(\mathbb{F}_v\right)$. Lifting to $P \in \P^n\left(K_v\right)$ via Hensel's lemma gives a local Campana point.
Otherwise, we have
$$
|\calD_0(\F_v)|=\frac{q_v^n-1}{q_v-1} > n\frac{q_v^{n-1}-1}{q_v-1} \geq |(\calD_0\cap(\calD_1\cup\cdots\cup\calD_n))(\F_v)|,
$$
hence there exists $P_v\in\calD_0(\F_v) \setminus \bigcup_{i=1}^n\calD_i\left(\mathbb{F}_v\right)$. Using Hensel's lemma, choose a lift $P\in D_0(K_v)$. Then $n_v(\calD_i,P)=\infty$ for $i=0$ and $0$ for all $i>0$. It follows that $P\in(\P^n_{\O_S},\calD)(\O_v)$.
\end{proof}
\begin{remark}
It would be of interest to determine easily checkable necessary conditions for the existence of local Campana points at a place.
\end{remark}
\subsection{Independent hyperplanes}
Let us now investigate CWA in the case $r \leq n$. 
\begin{proposition}\label{prop:projcampdense}
Let $D_0,\dots,D_r\subset \P^n$ be hyperplanes in general linear position. Then $(\P^n_{\calO_S},\calD)(\O_S)$ satisfies CWA.
\end{proposition}
\begin{proof}
By Section \ref{sec:reduction}, we may reduce to the case $D=H$, with $H$ as defined in \eqref{eqn:n+2}. Let $T\subset \Omega_K$ be a finite set of places. Write $T=T_{\textnormal{f}}\cup T_\infty$ where $T_\infty = T \cap S_{\infty}$ and $T_{\textnormal{f}} = T \setminus T_{\infty}$. For each $v\in T$, let $Q_v = [y_{v,0}:\cdots:y_{v,n}] \in (\P^n_{\calO_S},\calH)(\O_{v})$. Our goal is to find $P=[x_0:\cdots:x_n]\in\left(\P^n_{\calO_S},\calH\right)(\calO_S)$ simultaneously approximating each $Q_v$. Without loss of generality, we assume that $T_{\textnormal{f}}\neq\emptyset$. Further, we may assume without loss of generality that $y_{v,i} \neq 0$ for all $v \in T_{\textnormal{f}}$ and $i \in \{0,\dots,n\}$, since $\left(\P^n_{\O_S},\calH\right)\left(\calO_v\right) \setminus H\left(K_v\right)$ is dense in $\left(\P^n_{\O_S},\calH\right)\left(\O_v\right)$ for all $v \in T_{\textnormal{f}}$, so we may choose the coordinates $y_{v,i}$ such that $\min_i\{v\left(y_{v,i}\right)\} = 0$.

We now construct uniformisers $\pi_v\in\O_S$ for each $v\in T_{\textnormal{f}}$ inductively in the following way. Let $T_{\textnormal{f}}=\{v_1,\ldots,v_r\}$. First, use strong approximation \cite[II.15]{CF} to choose $\pi_{v_1}$ such that $v_j(\pi_{v_1})=0$ for each $j\neq1$. Then for each $i>1$, we use strong approximation again to construct $\pi_{v_i}$ such that
\begin{enumerate}[label=(\roman*)]
\item $v_j(\pi_{v_i})=0$ for each $j\neq i$;
\item $u(\pi_{v_i})=0$ whenever $u(\pi_{v_j})>0$ for any $u\not\in S_\infty,j<i$.
\end{enumerate}
This process gives us a set of uniformisers $\pi_v\in\O_S$ for each $v\in T_{\textnormal{f}}$ such that for all places $u \not\in S_\infty$, we have $u(\pi_v)>0$ for at most one $v\in T_{\textnormal{f}}$ .

For each $v\in T_{\textnormal{f}}$, write $y_{v,i}=u_{v,i}\pi_v^{e_{v,i}}$, where $u_{v,i}\in\O_{v}^\times$. For $v \in T_{\textnormal{f}} \setminus S$ and $0 \leq i \leq r$, note that either $e_{v,i} \geq m_i$ or $e_{v,i}=0$ since $Q_v \in (\P^n_{\calO_S}, \calH)(\O_{v})$. Let $N$ be a large positive integer. For each $v\in T_{\textnormal{f}}$, set $d_v=|(\O_v/\pi_v^{N})^\times|$ and $d=1+\prod_{v\in T_{\textnormal{f}}} d_v$. By iteratively applying strong approximation as before, we construct $\alpha_0,\ldots,\alpha_n\in \O_S$ inductively such that
\begin{align*}
v\left(\alpha_i - u_{v,i}\prod_{w\in T_{\textnormal{f}},w\neq v}\pi_w^{-e_{w,i}}\right)\geq N,\quad &v\in T_{\textnormal{f}}, \\
\left|\alpha_i^d\prod_{w \in T_{\textnormal{f}}}\pi_w^{e_{w,i}}-y_{v,i}\right|_v\leq1/N,\quad &v\in T_\infty,
\end{align*}
and
\[
v(\alpha_i)v(\pi_w)=v(\alpha_i)v(\alpha_j)=0,\quad v\notin T\cup S_{\infty},w\in T_{\textnormal{f}}, j < i.
\]
Observe that because $v_i(u_{v,i}\prod_{w\in T_{\textnormal{f}},w\neq v}\pi_w^{-e_{w,i}})=0$, we must have $v_i(\alpha_i)=0$ for each $i$. Therefore, $\alpha_i\in \O_v^\times$ and $\alpha_i^d-\alpha_i\equiv 0\bmod \pi_v^N$ by construction of $d$. For each $i \in \{0,\dots,n\}$, define the $S$-integer
\begin{equation*}
x_i=\alpha_i^d\prod_{v\in T_{\textnormal{f}}} \pi_v^{e_{v,i}}.
\end{equation*}
Then, for each $v \in T_{\textnormal{f}}$, we have
\begin{equation}  \label{eqn:valcong}
\begin{aligned}
v(x_i-y_{v,i}) & \geq \min\left\{v\left(\left(\alpha_i^d - \alpha_i\right)\prod_{w\in T_{\textnormal{f}}} \pi_w^{e_{w,i}}\right), v\left(\alpha_i\prod_{w\in T_{\textnormal{f}}} \pi_w^{e_{w,i}} -u_{v,i}\pi_v^{e_{v,i}}\right)\right\} \\
& \geq \min\left\{v\left(\alpha_i^d - \alpha_i\right) + e_{v,i}, e_{v,i} + N\right\} = e_{v,i} + N,
\end{aligned}
\end{equation}
where the last equality follows from the fact that $\alpha_i^d-\alpha_i\equiv 0\bmod \pi_v$. Hence, by choosing $N$ large enough, $x_i$ simultaneously approximates each $y_{v,i}$ arbitrarily well, so $P=[x_0:\cdots:x_n]$ approximates each $Q_v$. To show that $P$ is a Campana point, it suffices by \eqref{eqn:vadiccampana} to show that, for all $v \not\in S$, we have
\begin{equation}\label{eqn:Campana}
v(x_i)-\min_{0 \leq i \leq r} \{v(x_j)\} \in \Z_{\geq m_i} \cup \{0\}.
\end{equation}
For $v\in T\setminus S$, this follows from \eqref{eqn:valcong} since $Q_v \in (\P^n_{\calO_S},\calH)(\O_{v})$. Now let $v\notin S \cup T$. If $v(\pi_w)=0$ for each $w\in T_{\textnormal{f}}$, then $v\left(x_i\right) = dv\left(\alpha_i\right)$ for all $i$, so \eqref{eqn:Campana} follows since $d \geq \max_{0 \leq i \leq r}\{m_i\}$ for $N$ sufficiently large.
If $v(\pi_w)=a>0$, then $v(\pi_{w'})=0$ for any $w'\in T,w'\neq w$, and $v(\alpha_i)=0$ for all $i$, hence $v(x_i)=ae_{w,i}$ for all $i$ and \eqref{eqn:Campana} is satisfied.
\end{proof}
\subsection{An extra hyperplane}
We now consider the case $r=n+1$. We restrict to the case where the orbifold coefficients $m_i$ are all equal to some $m$. Consider the family of diagonal hypersurfaces of degree $m$ in $\P^{n+1}$,
\begin{equation}\label{eqn:defX}
\mathfrak{X}\colon a_0x_0^m+\dots+a_{n+1}x_{n+1}^m=0\subset \P^{n+1}\times\P^{n+1},
\end{equation}
together with its projection $\pi\colon \mathfrak{X}\to\P^{n+1}_{a_0,\ldots,a_{n+1}}$ to the first factor. For $P\in \P^{n+1}(K)$, let $\mathfrak{X}_P$ denote the fibre $\pi^{-1}(P)$. For a finite set $T \subset \Omega_K$, define
$$
V_T=\left\{P\in \pi(\mathfrak{X}(K)):\mathfrak{X}_P(K)\textnormal{ is dense in }\mathfrak{X}_P(\A_K^T)\right\}.
$$
Lastly, define the rational map
$$
\rho\colon \mathfrak{X}\dashrightarrow \P^n, \quad \left([a_0:\dots:a_{n+1}],[x_0:\dots:x_{n+1}]\right) \mapsto [a_0x_0^m:\dots:a_nx_n^m].
$$
\begin{lemma}\label{lem:campanaHypersurface}
Let $X$ be a smooth hypersurface in $\P^{n+1}_K$ defined by
$$
a_0x_0^m+\cdots+a_{n+1}x_{n+1}^m=0.
$$
Let $v \in \Omega_K \setminus S$ and assume that $X(K_v)\neq\emptyset$. 
\begin{enumerate}[label=(\roman*)]
\item If $v(a_i)=v(a_j)$ for all $i,j$, then $\rho(X(K_v))\subset (\P^n_{\calO_S},\calH)(\O_v)$.
\item If $v(m)=0$, then there exists $P\in X(K_v)$ such that $\rho(P)\in (\P^n_{\calO_S},\calH)(\O_v)$.
\end{enumerate}
\end{lemma}
\begin{proof}
Let $P=[x_0:\cdots:x_{n+1}]\in X(K_v)$. Note that $a_i\neq0$ for all $i$ since $X$ is smooth. Let $e=\min_i\{v(a_ix_i^m)\}$ and set
$$e_i=v(a_ix_i^m)-e,$$
taking $e_i=\infty$ if $x_i=0$.
\\
(i) Assume that $v(a_i)=v(a_j)$ for all $i,j$. Then $m\mid e_i$ or $e_i = \infty$ for each $i$, hence $\rho(P)\in(\P^n_{\O_S},\calH)(\O_v)$.
\\
(ii) Assume that $v\left(m\right) = 0$. Let $E=\{i: e_i=0\}$. Consider the equation
\begin{equation}\label{eqn:sumoverE}
\sum_{i\in E} \frac{a_ix_i^m}{\pi^{e}_v}y_i^m\equiv0\pmod v,
\end{equation}
where $\pi_v$ is a uniformiser in $\calO_v$. There is the obvious solution $y_i\equiv1\pmod v$. For $i\notin E$, set $x_i'=\pi_v^{e+|v(a_i)|+1}$. Now consider the equation
\begin{equation}\label{eqn:completesum}
\sum_{i\in E} \frac{a_ix_i^m}{\pi_v^{e}}y_i^m+\sum_{i\notin E} \frac{a_ix_i'^m}{\pi_v^e}=0.
\end{equation}
The reduction modulo $v$ is precisely \eqref{eqn:sumoverE}, and the solution $y_i\equiv1\pmod v$ can be lifted (as $v\left(m\right) = 0$) to a solution of \eqref{eqn:completesum}. Choose such a solution $\left(y_i\right)_{i \in E}$ to \eqref{eqn:completesum} and let $x_i'=x_iy_i$ for $i\in E$. Observe that $P'=[x_0':\dots:x_{n+1}']\in X(K_v)$. Set
$$
e_i'=v(a_ix_i'^m)-\min_j\{v(a_jx_j'^m)\}.
$$
We have $e'_i=0$ for $i\in E$ and $e'_i \geq m$ for $i\notin E$, hence $\rho(P')\in (\P^n_{\O_S},\calH)(\O_v)$.
\end{proof}
\begin{proposition}\label{prop:cwadiagonal}
Suppose that the $m_i$ are all equal to some $m\geq2$. If $V_T$ is dense in $\pi(\mathfrak{X}(\A_K))$, then $(\P^n_{\calO_S},\calH)(\O_S)$ is dense in $(\P^n_{\O_S},\calH)(\A^{T}_K)$. 
\end{proposition}
\begin{proof}
Let $T'$ be any finite set of places of $K$ disjoint from $T$. For each $v\in T'$, let $Q_v \in (\P^n_{\O_S},\calH)(\O_{v})$ if $v\notin S$ and $Q_v \in \P^n(K_v)$ otherwise.  It is easily seen that we may choose points $P_v\in \mathfrak{X}(K_v)$ such that $\rho(P_v) = Q_v$ and $\mathfrak{X}_{\pi(P_v)}$ is smooth. We will construct a point $M\in \mathfrak{X}(K)$ which is $v$-adically close to $P_v$ for each $v\in T'$ and $\rho(M)\in(\P^n_{\O_S},\calH)(\O_S)$.

Let $R_v=\pi(P_v)\in\P^{n+1}(K_v)$. Using the implicit function theorem, let $U_v\subset \P^{n+1}(K_v)$ be a small open neighbourhood around $R_v$ such that there is a section $\sigma_v\colon U_v\to \mathfrak{X}(K_v)$ to $\pi$ with $\sigma_v(R_v)=P_v$. Set $R_0=[-1:1:\dots:1]\in\pi(\mathfrak{X}(K)) \subset \P^{n+1}\left(K\right)$. As $V_T$ is dense in $\pi(\mathfrak{X}(\A_K))$, there is a point $R=[a_0:\dots:a_{n+1}]\in V_T$, which
\begin{enumerate}[label=(\roman*)]
\item lies in $U_v$ for each $v\in T'$, and
\item \label{item:2} approximates $R_0$ in $\P^{n+1}(K_v)$ for each $v\in T$ and each $v\notin T'$ dividing $m$.
\end{enumerate}
We now construct a point $\left(M_v\right)\in \mathfrak{X}_R(\A_K)$ as follows:
\begin{itemize}
\item For each $v\in T'$, let $M_v=\sigma_v(R)$.
\item For each $v\notin T'\cup S$, choose any point $M_v\in \mathfrak{X}_R(K_v)$ such that $\rho(M_v)\in(\P^{n}_{\O_S},\calH)(\O_v)$ using Lemma \ref{lem:campanaHypersurface} and (ii).
\item For each $v\in S\setminus T'$ choose any point $M_v\in \mathfrak{X}_R(K_v)$.
\end{itemize}
Let $E\subset\Omega_K$ be the finite subset of places $v$ such that $v(a_i)\neq v(a_j)$ for some $i,j$. Note that by (ii), we have $E\cap T=\emptyset$.

Since $\mathfrak{X}_R(K)$ is dense in $\mathfrak{X}_R(\A_K^T)$, there exists $M\in \mathfrak{X}_R(K)$ approximating $M_v$ in $\mathfrak{X}_R(K_{v})$ for each $v\in T'\cup E$. Note that, for $v \in T'$, the point $M$ is $v$-adically close to $P_v$, hence $\rho\left(M\right)$ is $v$-adically close to $Q_v$. Indeed, $M$ is $v$-adically close to $M_v = \sigma_v\left(R\right)$, which lies near $P_v$ as the open neighbourhood $U_v$ is small. To see that $\rho(M)\in(\P^n_{\O_S},\calH)(\O_S)$, first note that $\rho(M)\in(\P^n_{\O_S},\calH)(\O_v)$ for $v\in (T'\cup E)\setminus S$, since $M$ approximates $M_v$ and $\rho(M_v)\in(\P^n_{\O_S},\calH)(\O_v)$, since all points sufficiently close to $Q_v$ are local Campana points. For any other $v\notin S$, we have $v(a_i)=v(a_j)$ for all $i,j$, so we can apply Lemma \ref{lem:campanaHypersurface}. Finally, for all $v\in S \cap T'$, $\rho\left(M\right)$ approximates $\rho\left(M_v\right)$ which in turn approximates $Q_v$, hence $\rho(M)\in(\P^n_{\O_S},\calH)(\O_v)$, and for $v \in S \setminus T'$, there is no local Campana condition.
\end{proof}
\begin{corollary} \label{cor:logfanoquadric}
Let $D_0,\dots,D_{n+1}\subset \P^n$ be hyperplanes in general linear position. Suppose $m_i$ are all equal to some $m\geq2$.  If all smooth degree-$m$ hypersurfaces in $\P^{n+1}$ satisfy weak approximation, then $\left(\P^n_{\O_S},\calD\right)$ satisfies CWA.
\end{corollary} 
\begin{proof}
By Section \ref{sec:reduction}, we can assume that $D=H$ as in \eqref{eqn:n+2}. Let $\mathfrak{X}_0\subset \mathfrak{X}$ denote the open locus consisting of smooth fibres of $\pi$. Since $\pi(\mathfrak{X}_0(K))\subset V_{\emptyset}$ by assumption, it suffices to show that $\pi(\mathfrak{X}(K))$ is dense in $\pi(\mathfrak{X}(\A_K))$ by Proposition \ref{prop:cwadiagonal}. To see this, note that the projection of $\mathfrak{X}$ to the second factor $\P^{n+1}_{x_0,\ldots,x_{n+1}}$ in \eqref{eqn:defX} realises it as a $\P^n$-bundle over $\P^{n+1}$. Hence, $\mathfrak{X}$ is rational and smooth, so it satisfies weak approximation. Now $\mathfrak{X}(K)$ being dense in $\mathfrak{X}(\A_K)$ implies $\pi(\mathfrak{X}(K))$ is dense in $\pi(\mathfrak{X}(\A_K))$.
\end{proof}
\begin{remark}
By \eqref{eqn:logFano}, the log Fano condition for $(\P^n,D)$ is $m \leq n+1$, which is precisely the condition for the smooth fibres of $\pi$, which are hypersurfaces of degree $m$ in $\P^{n+1}$, to be Fano. Taking  $n=2$, this becomes $m \leq 3$. When $m=3$, then $\mathfrak{X}$ is the family of diagonal cubic surfaces. Work of Bright, Browning and Loughran \cite[Thm.~1.6]{BBL} implies that $V_{\emptyset}$ contains $0\%$ of the points in $\P^3(K)$ when ordered by height. However, this does not imply that $V_{\emptyset}$ cannot be dense in $\P^3(\A_K)$.
\end{remark} 

A common tool in studying weak approximation is the Brauer--Manin obstruction. Let $X$ be a variety over a number field $K$, and let $\Br(X)$ denote its (cohomological) Brauer group. There is a pairing
$$
\ev\colon X(\A_K)\times \Br(X) \rightarrow \Q / \Z, \quad
(P, \calA) \mapsto \ev_{\calA}(P).
$$
Denote the left kernel by $X(\A_K)^{\Br}$, which is sandwiched between two sets
$$X(K)\subset X(\A_K)^{\Br}\subset X(\A_K)$$
by class field theory. Since the pairing is continuous, the set $X(\A_K)^{\Br}$ acts as an obstruction to weak approximation and the existence of rational points. See \cite[\S8.2]{POO} for more details on this obstruction.
\begin{mydef} \label{def:WABMO}
We say that \emph{weak approximation with Brauer--Manin obstruction holds for $X$} if $X\left(K\right)$ is dense in $X\left(\A_K\right)^{\Br}$.
\end{mydef}
The Brauer--Manin obstruction to weak approximation is of particular interest in the case of Fano varieties due to the following conjecture of Colliot-Th\'el\`ene.
\begin{conjecture}[Colliot-Th\'el\`ene]\cite[p.~174]{CT}\label{conj:ct}
Let $X$ be a smooth projective variety over a number field $K$. If $X$ is geometrically rationally connected, then $X(K)$ is dense in $X(\A_K)^{\Br}$.
\end{conjecture}
\begin{note} \label{note:fano}
Fano varieties are geometrically rationally connected, since the property of being Fano is invariant under base change and a Fano variety over an algebraically closed field of characteristic zero is rationally connected \cite[p.~266]{POO}.
\end{note}
We now show that CWA almost always holds if we assume Conjecture \ref{conj:ct}.
\begin{corollary}\label{cor:projlogfano}
Let $D_0,\dots,D_{n+1}\subset \P^n$ be hyperplanes in general linear position. Suppose $m_i$ are all equal to some $m\geq2$. Assume Conjecture \ref{conj:ct} holds. Suppose that $(\P^n,D)$ is log Fano (which holds precisely when \ $m \leq n+1$).
\begin{enumerate}[label=(\roman*)]
\item If $(n,m)=(2,3)$, then for some place $v_0$, we have $(\P^2_{\O_S},\calD)(\O_S)$ is dense in $(\P^2_{\O_S},\calD)(\A^{\{v_0\}}_K)$.
\item Otherwise, $(\P^n_{\O_S},\calD)$ satisfies CWA.
\end{enumerate}
\end{corollary}
We first present a lemma which will be used in the proof. Given a variety $X$ over a field $k$, we denote by $\Br_0\left(X\right)$ the subgroup $\im\left(\Br k \rightarrow \Br X\right)$ of $\Br X$, where $\Br k \rightarrow \Br X$ is the map arising from the structure morphism of $X$.
\begin{lemma}\label{lem:projlogfano}
There exists a constant $B$ such that the following statement is true: let $X$ be any smooth diagonal cubic surface
$$a_0x_0^3+a_1x_1^3+a_2x_2^3+a_3x_3^3=0$$
such that there exists a finite place $v$ where $v(a_0)=1$ and $v(a_i)=0$ for all $i\neq 0$. Assume that $q_v > B$. Let $C \subset X$ be the curve given by $x_0 = 0$. Let $\calA \in \Br(X)[3]$ whose image in $\Br(X)/\Br_0(X)$ also has order $3$. Then the restricted evaluation map
$$
\ev_\calA|_C \colon C(K_{v})\hookrightarrow X(K_{v})\to\Br(K_{v})[3]
$$
is surjective.
\end{lemma}
\begin{proof}
We first show that, for $q_v$ large enough, the evaluation map
\begin{equation}\label{eqn:prolific}
\ev_\calA\colon X(K_{v})\to\Br(K_{v})[3]
\end{equation}
is surjective. Taking the $\mathcal{O}_v$-models $\mathcal{X}$ and $\mathcal{C}$ for $X$ and $C$ respectively defined by the given equations viewed over $\O_v$, note that $\mathcal{X}_{\F_v}$ is a projective cone over the cubic curve $\mathcal{C}_{\F_v}$. Then the surjectivity of \eqref{eqn:prolific} follows from either \cite[Thm.~6.5]{BRI15} or \cite[\S5]{CTKS}. We give an adaptation of the proof of the first cited theorem here. Denote by $\calX_{\F_v}^{\textnormal{sm}}$ the smooth locus of the special fibre of $\calX$. Consider the residue map $\partial\colon\Br(X)\to \mathrm{H}^1_{\textrm{\'et}}(\calX_{\F_v}^{\textnormal{sm}},\Q/\Z)$ which associates $\calA$ to an \'etale cover of $\calX_{\F_v}^{\textnormal{sm}}$ \cite[\S5.1]{BRI15}. Since $\calX_{\F_v}$ is a cone over the curve $\mathcal{C}_{\F_v}$, the image $\partial(\calA)$ has order $3$ \cite[Lem.~6.3]{BRI15} and thus corresponds to a degree-$3$ \'etale cover $Y\to\calX_{\F_v}^{\textnormal{sm}}$. Suppose $\beta\in\Br(K_v)[3]$ is such that the twisted cover $Y^\beta$ has an $\F_v$-point mapping to a point $P_v\in\calX_{\F_v}(\F_v)$. Then for any lift $P\in \calX(\O_v)$ of $P_v$, we have $\langle \calA,P\rangle=\beta$ \cite[Lem.~5.12]{BRI15}. By the Lang--Weil estimate, the existence of such $P$ is ensured for $q_v$ sufficiently large.

Since $\calX_{\F_v}^{\textnormal{sm}}$ is isomorphic to $\calC_{\F_v}\times\A^1$ and $\calC_{\F_v}$ is proper, there exists a degree-$3$ \'etale cover $D \to \calC_{\F_v}$ such that $Y^\beta=D\times_{\calC_{\F_v}}\calX_{\F_v}^{\textnormal{sm}}$. Together with the inclusion $\calC_{\F_v}\subset \calX_{\F_v}^{\textnormal{sm}}$, we have the diagram
\begin{equation*}
\begin{tikzcd}[column sep=0.5cm,row sep=0.5cm]
D \ar[r] \ar[d] & D\times \calX_{\F_v}^{\textnormal{sm}} \ar[r] \ar[d] & D \ar[d]\\
\calC_{\F_v} \ar[r] & \calX_{\F_v}^{\textnormal{sm}} \ar[r] & \calC_{\F_v},
\end{tikzcd}
\end{equation*}
where the composition of the horizontal arrows are isomorphisms. If $q_v$ is large enough, then $D$ has an $\mathbb{F}_v$-point by the Hasse--Weil bound \cite[Cor.~7.2.1]{POO} (note that $D$ is a curve of genus one). Lifting the image of such a point in $\calC_{\F_v}$ to a point $P\in \calC(\calO_v)$, we conclude using the last paragraph that $\langle \calA, P\rangle=\beta$.
\end{proof}
\begin{proof}[Proof of Corollary \ref{cor:projlogfano}]
By Section \ref{sec:reduction}, we can assume that $D=H$ as in \eqref{eqn:n+2}.

Let us first consider the case $n=2$. Then $(\P^2,D)$ is log Fano if and only if $m \leq 3$. The result for $m=2$ follows from Corollary \ref{cor:logfanoquadric}, since weak approximation with Brauer--Manin obstruction holds for quadrics, so let $m=3$. Fix a place $v_0$ such that $q_{v_0}>B$ for $B$ as in Lemma \ref{lem:projlogfano}. We show that $V_{\{v_0\}}$ is dense in $\pi(\mathfrak{X}(\A_K))$. Let $T'$ be any finite set of places not including $v_0$. Define $R_v,\sigma_v,U_v$ for $v\in T'$ and $R_0$ as in the proof of Proposition \ref{prop:cwadiagonal}. As shown in the proof of Corollary \ref{cor:logfanoquadric}, $\pi(\mathfrak{X}(K))$ is dense in $\pi(\mathfrak{X}(\A_K))$, so there exists $R=[a_0:a_1:a_2:a_3]\in \pi(\mathfrak{X}(K))$ with $a_0a_1a_2a_3\neq0$ such that
\begin{enumerate}[label=(\roman*)]
\item $R$ lies in $U_v$ for each $v\in T'$,
\item $R$ approximates $R_0$ in $\P^{n+1}(K_v)$ for each $v\notin T' \cup \{v_0\}$ dividing $3$,
\item $v_0(a_0)=1$ and $v_0(a_i)=0$ for all other $i>0$, and
\item $a_1/a_2$ is not a cube in $K$.
\end{enumerate}
The cubic surface $X:=\mathfrak{X}_R$ given by
$$
a_0x_0^3+a_1x_1^3+a_2x_2^3+a_3x_3^3=0
$$
contains a $K$-point since $R\in \pi(\mathfrak{X}(K))$. We claim that $\Br(X)/\Br_0(X)$ is either trivial or isomorphic to $\Z/3\Z$.

Let $\zeta$ be a cube root of unity. By \cite[Prop.~3.1]{U14} (see also \cite[\S1]{CTKS}), we have $\Br(X_{K(\zeta)})/\Br_0(X_{K(\zeta)})\isom \Z/3\Z$. Consider the restriction and corestriction maps
$$\res\colon \Br(X)/\Br_0(X)\to \Br(X_{K(\zeta)})/\Br_0(X_{K(\zeta)}),$$
$$\cores\colon \Br(X_{K(\zeta)})/\Br_0(X_{K(\zeta)})\to \Br(X)/\Br_0(X).$$
The composition $\cores\circ\res$ is multiplication by $[K(\zeta):K]=2$, which implies that the prime-to-$2$ part of $\Br(X)/\Br_0(X)$ injects to $\Z/3\Z$. Next, let $L=K(\sqrt[3]{a_0},\ldots,\sqrt[3]{a_3})$ so that $[L:K]$ is a power of $3$. Then $X_L$ is $L$-rational (see \cite[\S13.7]{HW08}), so $\Br(X_L)/\Br_0(X_L)=0$. Considering the restriction and corestriction maps again for this extension $L/K$, we find that the $2$-primary part of $\Br(X)/\Br_0(X)$ is trivial. Altogether, we deduce that the  above claim holds.

 Let $\calA\in\Br(X)[3]$ be a generator for this quotient (such an element exists since $\Br_0(X)\isom\Br(K)$ is divisible).
We now construct a point $\{M_v\}\in X(\A_K)$ as follows:
\begin{enumerate}[label=(\roman*)]
\item For each $v\in T'$, let $M_v=\sigma_v(R)$.
\item For each $v\notin T'\cup S\cup\{v_0\}$ choose any point $M_v\in X(K_v)$ such that $\rho(M_v)\in(\P^{n}_{\O_S},\calD)(\O_v)$ using Lemma \ref{lem:campanaHypersurface}.
\item For each $v\in S\setminus T'$ choose any point $M_v\in X(K_v)$.
\end{enumerate} 
Let $C\subset X$ be the curve given by $x_0=0$. If $\calA\notin\Br_0(X)$, then by Lemma \ref{lem:projlogfano}, the restricted evaluation map
$$\ev_\calA\colon C(K_{v_0})\hookrightarrow X(K_{v_0})\to\Z/3\Z$$
is surjective. (The reason for restricting to $C$ instead of $X$ is to guarantee that we obtain a Campana point, as we will see later.) Hence, there exists a point $M_{v_0}\in C(K_{v_0})$ such that
$$
\sum_{v\neq v_0}\inv_v\calA(M_v)+\inv_{v_0} \calA(M_{v_0})=0.
$$
If $\calA\in\Br_0(X)$, then choosing any point $M_{v_0}\in C(K_{v_0})$, the equation above holds by class field theory. Hence, in either case, the combined adelic point $\left(M_v\right)\in X(\A_K)$ is then orthogonal to $\Br(X)$. Since Conjecture \ref{conj:ct} is assumed to hold, there exists $M\in X(K)$ approximating $M_v$ for each (i) $v\in T'$ and (ii) $v$ such that $v(a_i)\neq v(a_j)$ for some $i,j$.  Similarly to the proof of Proposition \ref{prop:cwadiagonal}, we have $\rho(M)\in(\P^n_{\O_S},\calD)(\O_S)$ using openness of Campana points and (ii) above, together with Lemma \ref{lem:campanaHypersurface} (note that $\rho\left(M_{v_0}\right) = [0:a_1 b_1^3: a_2 b_2^3: a_3 b_3^3] \in (\mathbb{P}^3_{\mathcal{O}_S},\mathcal{H})(\mathcal{O}_{v_0})$).

Suppose now that $n\geq3$ and $(\P^n,D)$ is log Fano. Then for any $P\in\P^{n+1}(K)$ such that the fibre $\mathfrak{X}_P$ of $\rho$ is smooth, we have $\Br(\mathfrak{X}_P)=\iota^*\Br(K)$, where $\iota\colon \mathfrak{X}_P\to \Spec K$ is the structure morphism (see \cite[Prop.~A.1]{PV}). Moreover, as $\mathfrak{X}_P$ is Fano, it is geometrically rationally connected (see Note \ref{note:fano}). Hence Conjecture \ref{conj:ct} together with $\Br(\mathfrak{X}_P)=\iota^*\Br(K)$ imply that $\mathfrak{X}_P(K)$ is dense in $\mathfrak{X}_P(\A_K)$. Thus, $\mathfrak{X}_P$ satisfies weak approximation. We conclude that the hypotheses of Proposition \ref{prop:cwadiagonal} are satisfied with $T=\emptyset$, and so $(\P^n_{\calO_S},\calD)$ satisfies CWA.
\end{proof}
\subsection{Quadratic point on the projective line}
We conclude this section with a case where the orbifold divisor is not a collection of hyperplanes. This will be used in Section \ref{sec:fib} for applications to conic bundles. 

Recall that, given a curve $C$ over a number field $K$ and a closed point $P \in C$, we define the \emph{degree} of $P$ to be $\deg\left(P\right) = [K\left(P\right):K]$.

\begin{proposition}\label{prop:deg2P1}
Let $P\in\P^1$ be a closed point of degree $2$, and let $D=(1-1/m)P$ for some integer $m \geq 2$. Then $(\P^1_{\calO_S},\calD)$ satisfies CWA.
\end{proposition}
\begin{proof}
Since $S$ is arbitrary, we may assume by Lemma \ref{lem:twoModels} (by applying a suitable automorphism of $\P^1$) that $P$ is given by $x_0^2-ax_1^2=0$, where $a\in \O_S$ is square-free. At any place $v \not\in S$, a point $[x_0:x_1] \in \mathbb{P}^1\left(K_v\right)$ is a local Campana point if and only if
\begin{equation*}
v\left(x_0^2 - ax_1^2\right) - \min_i\{v\left(x_i^2\right)\} \in \Z_{\geq m} \cup \{0,\infty\}.
\end{equation*}

Let $T$ be a finite set of places. Let $Q_v\in (\P^1_{\calO_S},\calD)(\O_{v})$ for $v\in T\setminus S$ and $Q_v\in \P^1(K_v)$ for $v\in T\cap S$. Our goal is to find $[x_0:x_1] \in (\P^1_{\calO_S},\calD)(\O_S)$ approximating all $Q_v$ simultaneously. By slightly perturbing the $Q_v$ if need be, we may assume that they are not contained in $P$.

Write $Q_v=[y_{v,0}:y_{v,1}]$ where $y_{v,i}\in \O_v$. Set
$$
z_{v,0}=\frac{y_{v,0}}{y_{v,0}^2-ay_{v,1}^2},\quad z_{v,1}=\frac{-y_{v,1}}{y_{v,0}^2-ay_{v,1}^2}.
$$
Let $A=K[t_0,t_1,t_2,t_3]$ and $f_0,f_1\in A$ be defined by
$$
(t_0+\sqrt{a}t_1)^m(t_2+\sqrt{a}t_3)^{m+1}=f_0+\sqrt{a}f_1.
$$
Let $Z\subset \A^4 = \Spec A$ be defined by $f_0=f_1=0$.
We claim that $Z$ is of codimension $2$ in $\A^4$. Indeed, note that, if $f_0$ and $f_1$ have a non-constant common factor, then
$$
(t_0+\sqrt{a}t_1)^m(t_2+\sqrt{a}t_3)^{m+1} \quad \textrm{and} \quad (t_0-\sqrt{a}t_1)^m(t_2-\sqrt{a}t_3)^{m+1}
$$
must also have a non-constant common factor. This is the case if and only if
$$
(t_0+\sqrt{a}t_1)(t_2+\sqrt{a}t_3) \quad \textrm{and} \quad (t_0-\sqrt{a}t_1)(t_2-\sqrt{a}t_3)
$$
have a non-constant common factor, which is clearly false.
By \cite[Lem.~1.1]{W}, the open set $U=\A^4\setminus Z$ satisfies strong approximation off a place $v_0$, so, setting $\calU = \mathbb{A}^4_{\O_K} \setminus \calZ$ (with $\calZ$ the Zariski closure of $Z$ in $\A^4_{\O_K}$), the image
$$
U(K)\to \prod_{w\in S\cup T\setminus\{v_0\}}U(K_w)\times\prod_{w\notin S\cup T\cup\{v_0\}} \calU(\O_w)
$$
is dense. Choose $v_0\notin T$ to be a non-archimedean place with $a \notin K_{v_0}^2$ and $a \in \O_{v_0}^\times$. Then there exists $R=(z_0',z_1',y_0',y_1')\in U\left(K\right)$ that approximates $\left(z_{v,0},z_{v,1},y_{v,0},y_{v,1}\right) \in U\left(K_v\right)$ at every $v\in T$, and lies in $\calU(\O_w)$ for all non-archimedean places $w\notin S\cup T \cup\{v_0\}$. Let $x_i=f_i(R)$. The latter condition tells us that $\min(w(x_0),w(x_1))=0$ for $w\notin S\cup T \cup\{v_0\}$.

Since
$$(z_0'+\sqrt{a}z_1')(y'_0+\sqrt{a}y'_1)=z_0'y'_0+az_1'y'_1+\sqrt{a}(z_0'y'_1+z_1'y'_0)$$
and
$$z_{v,0}y_{v,0}+az_{v,1}y_{v,1}=1,\quad z_{v,0}y_{v,1}+z_{v,1}y_{v,0}=0,$$
we see that $x_i$ approximates $y_{v,i}$ in $K_v$. Observe that
\begin{align*}
x_0^2-ax_1^2&=(x_0+\sqrt{a}x_1)(x_0-\sqrt{a}x_1)\\
&=(z_0'+\sqrt{a}z_1')^m(z_0'-\sqrt{a}z_1')^m(y'_0+\sqrt{a}y'_1)^{m+1}(y'_0-\sqrt{a}y'_1)^{m+1}\\
&=({z'}_0^2-a{z'}_1^2)^m({y'}_0^2-a{y'}_1^2)^{m+1}.
\end{align*}
Since $\min(w(x_0),w(x_1))=0$ for any place $w\notin S\cup T \cup \{v_0\}$, we have
$$
w(x_0^2-ax_1^2)-\min\{w(x_i^2)\}=
\begin{cases} 0 \textrm{ or at least $m$} &\textnormal{if }w\notin S\cup\{v_0\},\\
0&\textnormal{if }w=v_0.
\end{cases}
$$
The case for $w=v_0$ above follows from the fact that $a\notin K_{v_0}^2$ and $a \in \O_{v_0}^\times$. Hence, the Campana conditions are satisfied and we have $[x_0:x_1]\in(\P^1_{\calO_S},\calD)(\O_S)$.
\end{proof}
\begin{remark}
Let $(\P^1,D)$ be as in Proposition \ref{prop:deg2P1}. While Lemma \ref{lem:twoModels} and Proposition \ref{prop:deg2P1} guarantee that any model $(\calX,\calF)$ of $(\P^1,D)$ satisfies CWA, it is possible that $(\calX,\calF)$ does not have local Campana points. To construct such an example, one can take a model where $\calX_v(\F_v) = \left(\calF_{\red}\right)_v(\F_v)$ but $\calF(\calO_v)=\emptyset$. Then for any $R\in X\left(K_v\right)$, we have $n_v(\calF,R)>0$, but it is impossible for $n_v(\calF,R)$ to be arbitrarily large by Hensel's lemma. We have the following concrete example.

Let $\calX\subset \P^2_{\Z}$ be the conic given by $x^2+y^2-9z^2=0$, and denote by $X \subset \P^2_{\Q}$ the generic fibre. Let $F=(1-\frac{1}{3}) Q$ with $Q = X \cap H$ for $H$ the line $x-y = 0$. Let $\calQ$ be the closure of $Q$ in $\calX$. It is easily seen that $\calQ_{\O_v} = \left(\calX \cap \calH\right)_{\O_v}$, where $\calH$ is the closure of $H$ in $\P^2_{\Z}$, cut out by the same equation over $\Z$. Thus, for $R = [a:b:c]\in\calX(\Z_3)$ with $a,b$ and $c$ coprime $3$-adic integers, we have
$$
n_3\left(\calQ,R\right) = v_3\left(a-b\right).
$$
By reducing the equation $a^2 + b^2 = 9c^2$ modulo $3$, it follows that $v_3\left(a\right)$ and $v_3\left(b\right)$ are both positive, thus $\min\{v_3\left(a\right),v_3\left(b\right),v_3\left(c\right)\} = v_3\left(c\right) = 0$ by the assumption of coprimality. Writing $a = 3a'$ and $b = 3b'$, we get $\left(a'\right)^2 + \left(b'\right)^2 = c^2$. If $v_3\left(a-b\right) \geq 3$, then $v_3\left(a' - b'\right) \geq 2$, but then $\left(A,B\right) = \left(a',2c\right)$ is a solution to $A^2 + B^2 \equiv 0 \pmod 9$, and it is easily seen that this congruence has no solutions with $3 \nmid B$.  
We deduce that $n_3\left(\calQ,R\right) \leq 2$, hence $\left(\calX,\calF\right)\left(\Z_3\right) = \emptyset$.
\end{remark}
\section{Fibrations and blowups}\label{sec:fib}
\subsection{CHP for fibrations}
In this section we explore the links between fibrations and the Campana Hilbert property.

The following result is a slight variant of \cite[Thm.~1.1]{BSFP}.
\begin{theorem} \label{thm:bsfp}
Let $f\colon X\to S$ be a dominant morphism of varieties over $K$. Let $B\subset S(K)$, $A\subset X(K)$. Suppose that the set $\{s\in B:f^{-1}(s)(K)\cap A\textnormal{ is not thin}\}$ is not thin. Then $A$ is not thin.
\end{theorem}

\begin{proof}
The result follows on replacing $X(K)$ by $A$ and $\Sigma$ by $\{s \in B: f^{-1}\left(s\right)\left(K\right) \cap A \textrm{ is not thin}\}$ in the proofs of \cite[Thm.~1.1]{BSFP} and \cite[Lem.~3.2]{BSFP}.
\end{proof}
\begin{lemma}\label{lem:fibration}
Let $\left(X,D\right)$ be a Campana orbifold over a number field $K$ with associated $\calO_S$-model $\left(\calX,\calD\right)$ for some finite set of places $S$ of $K$ containing $S_{\infty}$. Let $Y$ be a proper normal $K$-variety with $\calO_S$-model $\calY$. Let $\phi: X \rightarrow Y$ be a dominant morphism that extends to $\Phi: \calX \rightarrow \calY$. Write $D = \sum_{\alpha \in \calA}\epsilon_\alpha D_\alpha$, where each $D_\alpha$ is a prime divisor on $X$. Assume that each $\calD_\alpha$ is flat over $\calY$.

Let $P \in Y\left(K\right)$ with $Z = \phi^{-1}\left(P\right)$ normal. Denote by $\calZ$ the Zariski closure of $Z$ in $\calX$. Then $Z \cap D$ is an orbifold divisor on $Z$, the pair $\left(\calZ,\calZ \cap \calD\right)$ is an $\O_S$-model for $\left(Z,Z\cap D\right)$, and we have
$$
\left(\calZ,\calZ \cap \calD\right)\left(\O_S\right) = \left(\calX,\calD\right)\left(\O_S\right) \cap Z\left(K\right).
$$
\end{lemma}
\begin{proof}
We first show that $\left(\calZ,\calZ \cap \calD\right)$ is an $\O_S$-model for $\left(Z,Z\cap D\right)$. Flatness and properness of $\calZ$ are clear. Consider the following diagram, where the three vertical squares away from the front one are fibre products.
\begin{equation}\label{cd:cubes}
\begin{tikzcd}[column sep=0.2cm,row sep=0.2cm]
& D_\alpha \ar[rr] \ar[dd] && \calD_\alpha \ar[dd]\\
Z\cap D_\alpha \ar[ru] \ar[rr, crossing over] \ar[dd] && \calZ\cap \calD_\alpha \ar[ru] \\
& Y \ar[rr] && \calY\\
\Spec K \ar[rr] \ar[ru] && \Spec \calO_S \ar[ru] \ar[uu, crossing over, leftarrow]
\end{tikzcd}
\end{equation}
It follows that the front square is also a fibre product. Flatness of $\phi|_{D_{\alpha}}$ implies that $Z \cap D_{\alpha}$ is a divisor on $Z$. To show that $\calZ \cap \calD_\alpha$ has dense generic fibre, it suffices to show that each irreducible component of $\calZ \cap \calD_\alpha$ dominates $\Spec \calO_S$, since then \cite[Exercise~3.1.3]{LIU} implies that the generic fibre of each irreducible component is dense. Since $\Spec \calO_S$ is a Dedekind scheme, \cite[Prop.~4.3.9]{LIU} shows that this is equivalent to flatness of $\calZ \cap \calD_\alpha$ over $\Spec\calO_S$. Since $\calD_\alpha$ is flat over $\calY$, applying preservation of flatness under base change to the rightmost face of the cube on the right of \eqref{cd:cubes} shows that $\calZ \cap \calD_\alpha$ is flat over $\Spec\calO_S$.

It only remains to check that, for any $Q \in \left(Z \setminus D_{\textrm{red}}\right)\left(K\right)$ and $\alpha_v \in \calA_v$, we have
\begin{equation} \label{eqn:fiber}
n_v\left(\calZ \cap \calD_{\alpha_v},Q\right) = n_v\left(\calD_{\alpha_v},Q\right).
\end{equation}
This holds if and only if $\calQ_v^*\calD_{\alpha_v} = \calQ_v^*\left(\calZ \cap \calD_{\alpha_v}\right)$, which is clear since $\calQ \subset \calZ$.
\end{proof}
\subsection{Blowing up}
In this section we study the relationship between the Campana Hilbert property and blowing up.
\begin{lemma}\label{lem:thinBirInv}
Let $\varphi\colon X\dashrightarrow Y$ be a birational map of varieties over a field $F$. Let $B\subset X(F)$ be a non-thin subset. Then $\varphi(B)\subset Y(F)$ is not thin.
\end{lemma}
\begin{proof}
Denote by $U\subset X$ an open set on which $\varphi$ is an isomorphism, and suppose that $\varphi\left(B\right) \subset Y\left(F\right)$ is thin. Then $\varphi\left(B \cap U\right) \subset \varphi\left(B\right)$ is also thin, so there exists a finite collection of generically finite dominant morphisms of varieties $f_i: W_i \rightarrow Y$, $i=1,\dots,r$ of degree $\geq2$ such that $\varphi\left(B \cap U\right) \setminus \bigcup_{i=1}^r f_i\left(W_i\left(F\right)\right)$ is not dense in $Y$. Then $\varphi\left(B \cap U\right) \setminus \bigcup_{i=1}^r f_i\left(W_i\left(F\right)\right)$ is not dense in $\varphi\left(U\right) \isom U$. Let $W_i'=f_i^{-1}(\varphi(U))$ and $f_i' = f_i|_{W_i'}$, so that $\varphi^{-1}\circ f_i'\colon W_i'\to X$ are generically finite dominant morphisms of degree $\geq 2$.
Then $(B \cap U) \setminus \bigcup_{i=1}^r \left(\varphi^{-1} \circ f_i'\right)\left(W_i'\left(F\right)\right)$
is not dense in $U$, which implies that $B$ is thin, a contradiction.
\end{proof}
\begin{mydef}
Given a Campana orbifold $(X,D)$ with $\calO_S$-model $\left(\calX,\calD\right)$, we say that the set of Campana points $(\calX,\calD)(\O_S)$ is \emph{locally not thin at $P \in \left(\calX,\calD\right)\left(\O_S\right)$} if for all open subsets $U\subset (\calX,\calD)(\A^S_K)$ containing $P$, the set $(\calX,\calD)(\O_S)\cap U\subset X(K)$ is not thin. We say that $(\calX,\calD)(\O_S)$ is \emph{locally not thin} if it is locally not thin at all $P \in \left(\calX,\calD\right)\left(\O_S\right)$.
\end{mydef}
\begin{lemma}\label{lem:blowup}
Let $(X,D)$ be a Campana orbifold with $\O_S$-model $(\calX,\calD)$. Let $P\in \left(X\setminus D_{\textrm{red}}\right)(K)$. Define
$$
Y=\Bl_P X\xrightarrow{\rho} X,\quad \calY=\Bl_\calP \calX\xrightarrow{\widetilde{\rho}} \calX,
$$
$$
F=\rho^*D,\quad \calF=\widetilde{\rho}^{*}\calD.
$$
Then $(\calY,\calF)$ is an $\O_S$-model for $(Y,F)$. Suppose that for all finite subsets $S\subset T\subset \Omega_K$ large enough, the set of $\O_T$-Campana points $(\calY_T,\calF_T)(\O_T)$ is locally not thin. Then  $(\calX,\calD)(\O_S)$ is locally not thin away from $P$.
\end{lemma}
\begin{proof}
The fact that $\left(\calY,\calF\right)$ is an $\O_S$-model for $\left(Y,F\right)$ follows from commutativity of blowups and flat base change \cite[Prop.~8.1.12(iii)]{LIU}. Let $P_0\in (\calX,\calD)(\O_S)$ be a point not equal to $P$ and $U\subset (\calX,\calD)(\A^S_K)$ an open neighbourhood of $P_0$. We have $\widetilde{\rho}^{-1}(\calD)=\widetilde{\rho}^*\calD+\calE$, where $\calE\subset \calY$ is supported over the set of places $T$ over which $\calP$ and $\calD$ intersect, which is finite since $P \not\in D_{\textrm{red}}$. Note that $\rho^{-1}(P_0) \in (\calY_T,\calF_T)(\calO_T)$ since for each $v\notin T$, the point $\calP_v$ and $\calD_v$ are disjoint.

For each $v\in T\setminus S$, there exists an open neighbourhood $U'_v\subset X(K_v)$ of $P_0$ contained in $\left(\calX,\calD\right)\left(\calO_v\right)$. Shrinking $U$ if necessary, we can assume that its image under $X(\A^S_K)\to X(K_v)$ is contained in $U'_v$. Let $B$ be the preimage of $U$ under
$$
(\calY_T,\calF_T)(\O_T)\injects Y(\A^S_K) \to X(\A^S_K).
$$
Enlarging $T$ if necessary, we can assume that $B$ is not thin by the hypothesis. We claim that $\rho\left(B\right) \subset \left(\calX,\calD\right)\left(\O_S\right)$. Let $Q \in B$. The places where we must check that $\rho\left(Q\right)$ is a local Campana point are those where $\calQ$ meets $\calE$. These places are contained in $T$, and for $v \in T$, we have $\rho(Q) \in U'_v \subset \left(\calX,\calD\right)\left(\calO_v\right)$. Hence, $\rho(B)\subset(\calX,\calD)(\O_S)$, and the result follows as $\rho(B)$ is not thin by Lemma \ref{lem:thinBirInv}.
\end{proof}
\subsection{Proof of Theorem \ref{thm:dp}}
We are now ready to prove Theorem \ref{thm:dp}. In fact, we will prove the stronger result that the Campana points are locally not thin.
\begin{thm} \label{thm:dplnt}
Let $X$ be a del Pezzo surface of degree $d$ over a number field $K$ and let $L \subset X$ be a rational line. Let $D = \left(1-\frac{1}{m}\right)L$ for some integer $m \geq 2$. Let $\left(\calX,\calD\right)$ be an  $\calO_S$-model of $\left(X,D\right)$. Suppose that one of the following holds:
\begin{enumerate} [label=(\roman*)]
\item $d=4$ and $X$ has a conic fibration.
\item $d=3$.
\item $d =2$ and $X$ has a conic fibration.
\end{enumerate}
Then $(\calX,\calD)(\O_S)$ is locally not thin. In particular, since $L(K)\subset(\calX,\calD)(\O_S)$, $(\calX,\calD)$ satisfies CHP.
\end{thm}
\begin{proof}
Let $A=\left(\calX,\calD\right)(\O_S)$ and $P_0\in A$. We show that $(\calX,\calD)(\O_S)$ is locally not thin at $P$.

In case (ii), the divisor class $-K_X - L$ gives rise to a conic fibration of $X$, hence in all cases, $X$ has a conic fibration. Let $C$ be the divisor class of the fibres of the conic fibration, and let $U\subset X(\A_K^S)$ be an open subset containing $P_0$. Set $A^\circ=A\cap U$.

(i),(ii) Since $X$ contains a rational point and $d \geq 3$, it is unirational (see \cite[Rem.~9.4.11]{POO}). Then $X(K)$ is dense, and we can blow up $Y \xrightarrow{\rho} X$ away from $P_0$ and $D_{\textrm{red}}$ so that $Y$ is a del Pezzo surface over $K$ of degree $2$. Let $\calY$ be the $\O_S$-model for $Y$ obtained by blowing up $\calX$ at the corresponding $\O_S$-points. Note that the composition $Y\to X\to\P^1$ of $\rho$ with the conic fibration of $X$ induces a conic fibration of $Y$. Since $Y$ is a del Pezzo surface of degree $2$, it suffices by Lemma \ref{lem:blowup} to consider the case (iii).

(iii) The classes $C$ and $-2K_X - C$ give rise to two conic fibrations $\pi_i \colon X \rightarrow \P^1$, $i=1,2$ respectively, with $\pi_1^{-1}\left(P\right) \cdot \pi_2^{-1}\left(Q\right) = 4$ (see \cite[Proof~of~Lem.~3.5]{STR}) for any $P,Q\in\P^1(K)$. Without loss of generality (by swapping the two fibrations if necessary), we have either (a) $L\cdot C=1$ or (b) $L\cdot C=2$ by the adjunction formula \cite[Exercise~V.1.3]{HART}. By spreading out, there is a finite subset $S\subset T\subset \Omega_K$ such that each $\pi_i$ extends over $\Spec\O_T$ to a flat morphism $\Pi_i\colon \calX_T\to \P^1_{\O_T}$. By enlarging $T$ if necessary, we may assume that $\calD_{\textrm{red}}$ is flat over $\Spec \O_T$.

Using weak approximation on $\P^1$, choose a point $P\in \P^1(K)\cap \pi_1(U)$ such that $\pi_1^{-1}(P)$ is smooth. Set $C_1 = \pi_1^{-1}(P)$, and denote by $\calC_1$ the closure of $C_1$ in $\calX_T$. Further enlarging $T$ if necessary, we may assume that $\mathcal{D}_{\textrm{red}}$ is flat relative to $\Pi_1$, and that it is flat relative also to $\Pi_2$ in case (a). Shrinking $U$ if necessary, we may assume that for each $v\in T\setminus S$, the image of the projection $U\to X(K_v)$ lies in $(\calX,\calD)(\O_v)$. In other words, any Campana $\O_T$-point lying inside $U$ is automatically a Campana $\O_S$-point, so
\begin{equation}\label{eqn:ot_os}
(\calX_T,\calD_T)(\O_T)\cap U=A^\circ.
\end{equation}

Let
\begin{align*}
Z_1 & =\{ P \in \P^1(K) : \pi_1^{-1}(P)\textnormal{ smooth and }\pi_1^{-1}(P)(K)\cap A^\circ \neq\emptyset \}, \\
Z_2 & =\{ P \in \P^1(K) : \pi_1^{-1}(P)\textnormal{ smooth and }\pi_1^{-1}(P)(K)\cap A^\circ \textnormal{ is not thin in }\pi_1^{-1}(P) \}.
\end{align*}
To conclude that $A^\circ$ is not thin, it suffices to show that $Z_2$ is not thin by Theorem \ref{thm:bsfp}. If $P\in Z_1$, then by Lemma \ref{lem:fibration}, $\pi_1^{-1}(P)\cap A^{\circ}$ is the same as the set of Campana $\O_T$-points on a model of $(\pi_1^{-1}(P), \pi_1^{-1}(P)\cap D)$ intersected with $U$. Note that Proposition \ref{prop:projcampdense} and Proposition \ref{prop:deg2P1} imply that the latter intersection is not thin in $\pi_1^{-1}(P)$ (the choice of model here does not matter by Lemma \ref{lem:twoModels}). Hence, $\pi_1^{-1}(P)\cap A^{\circ}$ is not thin, so $Z_1=Z_2$. Then it suffices to show that $Z_1$ is not thin.

By Lemma \ref{lem:fibration} and the above, $(\calC_1,\calC_1\cap\calD_T)$ is an $\O_T$-model for $(C_1,C_1\cap D_T)$. Then Lemma \ref{lem:twoModels} and Proposition \ref{prop:projcampdense} (resp.\ Proposition \ref{prop:deg2P1}) show that $(\calC_1,\calC_1\cap\calD_T)$ satisfies CWA in case (a) (resp.\ case (b)). By Lemma \ref{lem:fibration}, $(\calC_1,\calC_1\cap\calD_T)(\O_T)=(\calX_T,\calD_T)(\O_T)\cap C_1(K)$, giving infinitely many Campana $\O_T$-points lying on $U$. By \eqref{eqn:ot_os}, this gives infinitely many points in $\pi_2(A^\circ)\subset \P^1(\A_K^S)$.

Let $B \subset \P^1(K)$ be thin. We adapt the method in the proof of \cite[Prop.~3.3]{STR} to construct a fibre $C_2$ of $\pi_2$ such that $\pi_1(C_2\left(K\right) \cap A^\circ) \not\subset B$. From this, we can deduce that $Z_1$ is not contained in $B$, and so $Z_1$ is not thin. We can exclude type I thin 
sets from $B$ since any type I thin set is contained in a type II thin set.

Write $B=\cup_{i=1}^n f_i(Y_i\left(K\right))$ for $f_i\colon Y_i\to \P^1$ finite covers of degree $\geq2$. Let $\Sigma_i\subset \P^1$ denote the finitely many branch points of $f_i$, and let $\Lambda$ denote the union of the branch loci of the morphisms
$$
\pi_2|_{\pi_1^{-1}\left(P\right)} \colon \pi_1^{-1}\left(P\right) \rightarrow \P^1,\ P \in \bigcup_{i=1}^n\Sigma_i,
$$
which is a finite set. Thus, we can take $Q \in \pi_2\left(A^\circ\right) \setminus \Lambda$ such that $\pi_2^{-1}\left(Q\right)$ is smooth, and set $C_2 = \pi_2^{-1}\left(Q\right)$. The proof of \cite[Prop.~3.3]{STR} shows that the branch locus of $\pi_1|_{C_2} \colon C_2 \rightarrow \P^1$ is disjoint with the branch locus of each of the morphisms $f_i \colon Y_i \rightarrow \P^1$, hence the fibre products $C_2 \times_{\P^1}Y_i$ are all covers of degree $\geq 2$ of $C_2$. Since rational points of the fibre product are in bijection with rational points of both curves mapping to the same point on $\mathbb{P}^1$, and since $C_2$ has the Hilbert property, the set $\{P\in C_2\left(K\right) : \pi_1(P)\in B\} \subset C_2\left(K\right)$ is thin. It now suffices to show that $C_2(K)\cap A^\circ$ is not thin. Let $\calC_2$ denote the closure of $C_2$ in $\calX_T$.

In case (a), we have
$$
C_2(K)\cap A^\circ=C_2(K)\cap (\calX_T,\calD_T)(\O_T)\cap U=(\calC_2,\calC_2\cap\calD_T)(\calO_T)\cap U,
$$
where the first equality is by \eqref{eqn:ot_os} and the second equality is by Lemma \ref{lem:fibration}, recalling that $\calD_{\textrm{red}}$ may be assumed flat relative to $\Pi_2$. Note also that this set is non-empty as it contains $Q$. Lemma \ref{lem:twoModels} implies that $(\calC_2,\calC_2\cap\calD_T)$ satisfies CWA, so it follows that $(\calC_2,\calC_2\cap\calD_T)(\calO_T)\cap U$ is not thin.

In case (b), since $C_2$ does not meet $L$, we may assume upon enlarging $T$ that $\left(\calC_2\cap \calD_{\textrm{red}}\right)_T = \emptyset$. Hence, $C_2(K) \cap A^\circ=C_2\left(K\right) \cap U$. Since $C_2$ satisfies weak approximation, $C_2(K)\cap U$ is not thin in $C_2$.
\end{proof}

\end{document}